\documentclass [12pt,a4paper,reqno]{amsart}
 \input amssymb.sty
\newcommand{\ef}{\end{equation}}
\chardef\bslash=`\\ % p. 424, TeXbook
\newtheorem{thm}{Theorem} [section]
\newtheorem*{thm*}{Theorem}
\newtheorem{cor}[thm]{Corollary}

\newtheorem{lem}[thm]{Lemma}

\newtheorem{defn}[thm]{Definition}%[section]

\newtheorem{acknowledgment*}[thm] {Acknowledgment}

 \theoremstyle{remark}

 \newtheorem{remark}[thm]{Remark}

\newcommand{\thmref}[1]{Theorem~\ref{#1}}

\newcommand{\lemref}[1]{Lemma~\ref{#1}}

 \renewcommand{\sectionmark}[1]{}

\newcommand{\doe}{\overset{\text{def}}{=}}
 
\newcommand{\loc} {\operatorname{loc}}

 \date{}

\begin{document}

\title[A criterion for compactness in $L_p(\mathbb R)$]
{A criterion for compactness in $L_p(\mathbb R)$ of the
resolvent of \\
 the maximal Sturm-Liouville operator of general form}
\author[N.A. Chernyavskaya]{N.A. Chernyavskaya}
\address{Department of Mathematics and Computer Science, Ben-Gurion
University of the Negev, P.O.B. 653, Beer-Sheva, 84105, Israel}
\author[L.A. Shuster]{L.A. Shuster}
 \address{Department of Mathematics,
 Bar-Ilan University, 52900 Ramat Gan, Israel}
 \email{miriam@macs.biu.ac.il}


\begin{abstract}
We consider  the equation
$$
-(r(x)y'(x))'+q(x)y(x)=f(x),\quad x\in \mathbb R
$$
where $f\in L_p(\mathbb R),$\ $p\in(1,\infty)$ and
$$
  r>0, \  q\ge0, \  \frac{1}{r}\in L_1^{\loc}(\mathbb R), \ q\in
 L_1^{\loc}(\mathbb R),
$$
 $$
 \lim_{|d|\to\infty}\int_{x-d}^x\frac{dt}{r(t)}\cdot \int_{x-d}^x q(t)dt=\infty,\quad x\in\mathbb R.
$$
 We assume that this equation is correctly solvable in
 $L_p(\mathbb R).$ Under these assumptions, we study the problem
 on compactness of the resolvent $\mathcal L_p^{-1}:L_p(\mathbb
 R)\to L_p(\mathbb R)$ of the maximal continuously invertible
 Sturm-Liouville operator $\mathcal L_p:\mathcal D_p(\mathbb R)\to
 L_p(\mathbb R).$
 Here
 $$\mathcal L_py=-(ry')'+qy,\qquad y\in\mathcal D_p$$
 $$\mathcal D_p=\{y\in L_p(\mathbb R): y,ry'\in\mathcal A
 C^{\loc}(\mathbb R),\ -(ry')'+qy\in L_p(\mathbb R)\}.$$
 For the compact operator $\mathcal L_p^{-1}: L_p(\mathbb R)\to
 L_p(\mathbb R)$, we obtain two-sided sharp by order estimates of
 the maximal eigenvalue.
\end{abstract}

 \maketitle

\baselineskip 20pt

\section{Introduction}\label{introduction}
\setcounter{equation}{0} \numberwithin{equation}{section}

In the present paper, we consider the equation
\begin{equation}\label{1.1}
-(r(x)y'(x))'+q(x)y(x)=f(x),\quad x\in \mathbb R
\end{equation}
 where $f\in L_p(\mathbb R),$\ $(L_p(\mathbb R):=L_p),$\ $p\in(1,\infty)$ and
 \begin{equation}\label{1.2}
  r>0, \  q\ge0, \  r^{-1}\in  L_1^{\loc}(\mathbb R), \ q\in
 L_1^{\loc}(\mathbb R)\quad \left(r^{-1}:\equiv\frac{1}{r}\right),
 \end{equation}
 \begin{equation}\label{1.3}
 \lim_{|d|\to\infty}\int_{x-d}^x\frac{dt}{r(t)}\cdot\int_{x-d}^x
 q(t)dt=\infty,\quad x\in\mathbb R.
 \end{equation}

Our general goal consists in finding criteria for compactness of
the resolvent of equation \eqref{1.1}. To state the problem more
precisely, we need the following definitions and restrictions.

Here and in the sequel, by a solution of equation \eqref{1.1}, we
mean any function $y$   absolutely continuous together with $ry'$
and satisfying \eqref{1.1} almost everywhere on $\mathbb R.$ We
say that equation \eqref{1.1} is correctly solvable in a given
space $L_p,$\ $p\in[1,\infty)$ if the following assertions hold
(see \cite[Ch.III, \S6, no.2]{1}):
\begin{enumerate}
 \item[I)] for every function $f\in L_p$, there exists a unique
 solution of \eqref{1.1}, $y\in L_p;$

\item[II)] there exists an absolute constant $c(p)\in(0,\infty)$
such that the solution of \eqref{1.1}, $y\in L_p,$ satisfies the
inequality
  \begin{equation}\label{1.4}
   \|y\|_p\le c(p)\|f\|_p,\qquad \forall f\in L_p\quad (\|f\|_p:=\|f\|_{L_p}).
 \end{equation}
\end{enumerate}

See \cite{2} and \S2 below for precise conditions that guarantee
I)--II).  In the sequel, for brevity, this is referred to as
``problem I)--II)" or ``question on I)--II)". It is easy to see
that the  problem I)--II) can be reformulated in different terms
(see \cite{2,3}).

To this end, let us introduce the set $\mathcal D_p$ and the
operator $\mathcal L_p:$
$$\mathcal D_p=\{y\in L_p: y,ry'\in L_p,\quad
-(ry')'+qy\in L_p\},$$
$$\mathcal L_py=-(ry')'+qy,\quad y\in\mathcal D_p.$$
(Here $\mathcal A C^{\loc}(\mathbb R)$ is the set of functions
absolutely continuous on every finite segment.) The linear
operator $\mathcal L_p$ is called the maximal Sturm-Liouville
operator, and problem   I)--II) is obviouisly equivalent to the
problem on existence and boundedness of the operator $\mathcal
L_p^{-1}: L_p\to L_p $ (see \cite{3}).

We can now give a precise statement of the problem studied in the
present paper:

\textit{To find minimal additional requirements  to \eqref{1.2}
and \eqref{1.3} to the functions $r$ and $q$ under which, together
with I)--II),  the following condition III) also holds (``problem
I)--III)" or ``question on I)--III)"):}

III) \textit{for a given $p\in(1,\infty)$ the operator $\mathcal
L_p^{-1}: L_p\to L_p$ is compact.}

The main goal of the present paper is an answer to the question on
I)--III).

For the reader's convenience we outline the structure of the
paper. In \S2 we collect the preliminaries necessary for
exposition; \S3 contains a list of all results of the paper
together with comments; \S4 contains the proofs; in \S5 we present
examples of applications of our results to a concrete equation;
and, finally, \S6 contains the proofs of some technical
assertions.

\section{Preliminaries}

\begin{thm}  \cite{4}  \label{thm2.1}
 Suppose that conditions \eqref{1.2} and
\begin{equation}\label{2.1}
\int_{-\infty}^x q(t)dt>0,\qquad \int_x^\infty q(t)dt>0,\qquad
x\in\mathbb R \end{equation} hold. Then the equation
\begin{equation}\label{2.2}
(r(x)z'(x))'=q(x)z(x),\qquad x\in\mathbb R
\end{equation}
has a fundamental system of solutions (FSS) with the following
properties:
\begin{equation}\label{2.3}
v(x)>0,\ u(x)>0,\quad v'(x)\ge0,\quad u'(x)\le0,\qquad x\in\mathbb
R,
\end{equation}
\begin{equation}\label{2.4}
r(x)[v'(x)u(x)-u'(x)v(x)]=1,\qquad x\in\mathbb R,
\end{equation}
\begin{equation}\label{2.5}
\lim_{x\to-\infty}\frac{v(x)}{u(x)}=\lim_{x\to\infty}\frac{u(x)}{v(x)}=0,
\end{equation}
\begin{equation}\label{2.6}
\int_{-\infty}^0\frac{dt}{r(t)u^2(t)}<\infty,\ \int_0^
\infty\frac{dt}{r(t)v^2(t)}<\infty,\
\int_{-\infty}^0\frac{dt}{r(t)v^2(t)}=\int_0^\infty\frac{dt}{r(t)u^2(t)}=\infty.
\end{equation}
\end{thm}

Moreover, properties \eqref{2.3}--\eqref{2.6} determine the FSS
$\{u,v\}$ uniquely up to constant mutually inverse factors.

\begin{cor}  \cite{4} \label{cor2.2}
Suppose that conditions \eqref{1.2} and \eqref{2.1} hold. Then
equation \eqref{2.2} has no solutions $z\in L_p$ apart from
$z\equiv 0.$
\end{cor}

The FSS from \thmref{thm2.1} is denoted below by $\{u,v\}$.

\begin{thm} \cite{4,5}   \label{thm2.3}
For the FSS $\{u,v\}$ we have the Davies-Harrell representations
\begin{equation}\label{2.7}
u(x)=\sqrt{\rho(x)}\exp\left(-\frac{1}{2}\int_{x_0}^x\frac{d\xi}{r(\xi)\rho(\xi)}\right),\quad
v(x)=\sqrt{\rho(x)}\exp\left(\frac{1}{2}\int_{x_0}^x\frac{d\xi}{r(\xi)\rho(\xi)}\right)
\end{equation}
where $x\in\mathbb R,$ $\rho(x)=u(x)v(x),$ $x_0$ is  a unique
solution of the equation $u(x)=v(x)$ in $\mathbb R.$ Furthermore,
for the Green function $G(x,t)$ corresponding to equation
\eqref{1.1}:
\begin{equation}\label{2.8}
G(x,t)=\begin{cases} u(x)v(t),\quad & x\ge t\\ u(t)v(x), \quad &
x\le t\end{cases}\end{equation} and for its ``diagonal value"
$G(x,t)\big|_{x=t}=\rho(x)$, we have the following representation
\eqref{2.9} and equalities \eqref{2.10}:
\begin{equation}\label{2.9}
G(x,t)=\sqrt{\rho(x)\rho(t)}\exp\left(-\frac{1}{2}\left|\int_x^t\frac{d\xi}{r(\xi)\rho(\xi)}\right|\right),\quad
x,t\in\mathbb R,\end{equation}
\begin{equation}\label{2.10}
\int_{-\infty}^0\frac{d\xi}{r(\xi)\rho(\xi)}=\int_0^\infty\frac{d\xi}{r(\xi)\rho(\xi)}=\infty.\end{equation}
\end{thm}

\begin{remark}\label{rem2.4}
Representations \eqref{2.7} and \eqref{2.8} are given in \cite{5}
for $r\equiv1$ and in \cite{4} for $r\not\equiv1.$ See \cite{4}
for equalities \eqref{2.10}. Throughout the sequel conditions
\eqref{1.2}--\eqref{1.3} are assumed to be satisfied (if not
stated otherwise) without special mentioning.
\end{remark}

\begin{lem}  \cite{4}  \label{lem2.5}
  For every given $x\in\mathbb R$ each of the following
equations \begin{equation}\label{2.11}
\int_{x-d}^x\frac{dt}{r(t)}\cdot \int_{x-d}^xq(t)dt=1,\qquad
\int_x^{x+d}\frac{dt}{r(t)}\cdot\int_x^{x+d}q(t)dt=1\end{equation}
in $d\ge0$ has a unique finite positive solution. Denote them by
$d_1(x)$ and $d_2(x),$ respectively. For $x\in\mathbb R$ we
introduce the following functions:
\begin{equation}
\begin{aligned}\label{2.12}
\varphi(x)&=\int_{x-d_1(x)}^x\frac{dt}{r(t)},\qquad
\psi(x)=\int_x^{x+d_2(x)}\frac{dt}{r(t)},\\
  h(x)&=\frac{\varphi(x)\psi(x)}{\varphi(x)+\psi(x)}\
\left(\equiv\left(\int_{x-d_1(x)}^{x+d_2(x)}q(t)dt\right)^{-1}\right).
\end{aligned}
\end{equation}
\end{lem}

\begin{thm} \cite{4}  \label{thm2.6}
For $x\in\mathbb R$ the following inequalities hold:
\begin{equation}
\begin{aligned}\label{2.13}
2^{-1}v(x)\le(r(x)v'(x))\varphi(x)\le 2v(x)\\
2^{-1}u(x)\le(r(x)|u'(x)|)\psi(x)\le 2u(x)
\end{aligned}
\end{equation}
\begin{equation}\label{2.14}
2^{-1}h(x)\le\rho(x)\le 2h(x).\end{equation}
\end{thm}

\begin{cor} \cite{4} \label{cor2.7}
Let $r\equiv1.$ For every given $x\in\mathbb R$ consider the
following equation:
\begin{equation}\label{2.15}
d\cdot\int_{x-d}^{x+d}q(t)dt=2 \end{equation}
 in $d\ge0.$ Equation \eqref{2.15} has a unique finite positive
 solution. Denote it by $\tilde d(x).$ We have the inequalities:
 \begin{equation}\label{2.16}
 4^{-1}\cdot\tilde d(x)\le \rho(x)\le 3\cdot 2^{-1}\tilde
 d(x),\qquad x\in\mathbb R.\end{equation}
\end{cor}

\begin{remark}\label{rem2.8}
Two-sided sharp by order a priori estimate of type \eqref{2.13}
first appear in \cite{6} (for $r\equiv1$ and under some additional
requirements to $q).$ Under conditions \eqref{1.2} and
$\inf\limits_{x\in\mathbb R} q(x)>0,$ estimates similar to
 \eqref{2.13}, with other more complicated auxiliary functions,
 were given in \cite{7}. Sharp by order estimates of the function
 $\rho$ were first obtained in \cite{8} (under some additional
 requirements to $r$ and $q$). Therefore, we call inequalities of
 such type Otelbaev inequalities. Note that in \cite{8} auxiliary
 functions more complicated than $h$ and $\tilde d$ were used.
 The function $\tilde d$ was introduced by M. Otelbaev (see
 \cite{9}).

 Throughout the sequel we denote by $c, c(p),\dots$ absolute
 positive constants which are not essential for exposition and may
 differ even within a single chain of computations. We write
 $\alpha(x)\asymp\beta(x),$ $x\in (a,b)$ if positive functions
 $\alpha$ and $\beta$ defined in $(a,b)$ satisfy the inequalities
 $$c^{-1}\cdot\alpha(x)\le\beta(x)\le c\alpha(x),\qquad x\in
 (a,b).$$
 \end{remark}

 \begin{lem}  \cite{10} \label{lem2.9} For $x\in\mathbb R$ we have
 the inequality
 \begin{equation}\label{2.17}
 r(x)|\rho'(x)|<1.
 \end{equation}
 In addition, the inequality $m<1$ where
 \begin{equation}\label{2.18}
 m=\sup_{x\in\mathbb R}r(x)|\rho'(x)|\end{equation}
 holds if and only if $\varphi(x)\asymp \psi(x),$\ $x\in\mathbb
 R.$\end{lem}

 We also introduce  a new auxiliary function $s$ and the function
 $d$ already known from \cite{4}. The properties of the functions are
 similar, and therefore for brevity we present them together. See
 \cite{4} for the proofs for $d,$ and \S6 below for the proofs for
 $s.$

 \begin{lem} %[\cite{4} [\S6   below{]}]
\cite [\S6   below]{4}
  \label{lem2.10}   For every
 $x\in\mathbb R$ each of the equations
 \begin{equation}\label{2.19}
 \int_{x-d}^{x+d}\frac{dt}{r(t)h(t)}=1,\qquad
 \int_{x-s}^{x+s}\frac{dt}{r(t)\rho(t)}=1
 \end{equation}
 in $d\ge0$ and $s\ge0$ has a unique finite positive solution.
 Denote the solutions of \eqref{2.19} by $d(x)$ and $s(x)$,
 respectively. The functions $d(x)$ and $s(x)$ are continuous for
 $x\in\mathbb R.$\end{lem}

  \begin{lem} \cite [\S6   below]{4}  \label{lem2.11}   %{\rm [\S6   below]}\
  For
  $x\in\mathbb R$, $t\in[x-\varepsilon d(x),x+\varepsilon d(x)]$
  $(t\in[x-\varepsilon s(x),x+\varepsilon s(x)])$ and
  $\varepsilon\in[0,1]$, we have the inequalities:
  \begin{equation}\label{2.20}
  (1-\varepsilon)d(x)\le d(t)\le(1+\varepsilon)d(x),\end{equation}
  \begin{equation}\label{2.21}
  ((1-\varepsilon)s(x)\le s(t)\le(1+\varepsilon)s(x)).\end{equation}
  In addition, we have the equalities:
  \begin{equation}\label{2.22}
  \lim_{x\to-\infty}(x+d(x))=-\infty,\qquad
  \lim_{x\to\infty}(x-d(x))=\infty,\end{equation}
   \begin{equation}\label{2.23}
  \Big(\lim_{x\to-\infty}(x+s(x))=-\infty,\quad
  \lim_{x\to\infty}(x-s(x))=\infty\Big).\end{equation}
    \end{lem}

    \begin{defn}  \cite{19}  \label{defn2.12}
    Suppose we are given $x\in\mathbb R,$ a positive and
    continuous function $\varkappa(t)$ for $t\in\mathbb R$, a
    sequence $\{x_n\}_{n\in\mathbb N'},$ $\mathbb
    N'=\{\pm1,\pm2,\dots\}.$ Consider segments
    $\Delta_n=[\Delta_n^-,\Delta_n^+],$
    $\Delta_n^{\pm}=x_n\pm\varkappa(x_n).$ We say that the
    segments
    $\{\Delta_n\}_{n=1}^\infty\left(\{\Delta_n\}_{n=-\infty}^{-1}\right)$
    form an $\mathbb R(x,\varkappa)$-covering of $[x,\infty)$
    $((-\infty,x])$ if the following requirements hold:
    \begin{enumerate}\item[1)] $\Delta_n^+=\Delta_{n+1}^-$ for
    $n\ge1$\quad
    $(\Delta_{n-1}^+=\Delta_n^-$ for $n\le-1)$,
\item[2)] $\Delta_1^-=x$ $(\Delta_{-1}^+=x),$\quad
    $\bigcup\limits_{h\ge1}\Delta_n=[x,\infty)$\quad
    $\Big(\bigcup\limits_{n\le-1}\Delta_n=(-\infty,x]\Big).$
    \end{enumerate}
    \end{defn}

    \begin{lem}  \cite{19}  \label{lem2.13}
Suppose that for a positive and continuous function $\varkappa(t)$
for $t\in\mathbb R$, we have the relations
\begin{equation}\label{2.24}
\lim_{t\to\infty}(t-\varkappa(t)=\infty\quad
\Big(\lim_{t\to-\infty}(t+\varkappa(t)))=-\infty\Big).
\end{equation}
Then for every $x\in\mathbb R$ there is an $\mathbb
R(x,\varkappa)$-covering of $[x,\infty)(\mathbb
R(x,\varkappa)$-covering of $(-\infty,x])$.
\end{lem}

\begin{remark}\label{rem2.14}
If for some $x\in\mathbb R$ there exist $\mathbb
R(x,\varkappa)$-coverings of both $[x,\infty)$ and $(-\infty,x]$,
then their union will be called an $\mathbb
R(x,\varkappa)$-covering of $\mathbb R.$
\end{remark}

\begin{lem} \cite [\S6   below]{4} \label{lem2.15} For every
$x\in \mathbb R$ there exist $\mathbb R(x,d)$ and $\mathbb
R(x,s)$-coverings of $\mathbb R.$\end{lem}

\begin{remark}\label{rem2.16}
Assertions of the type in \lemref{lem2.15} and estimates of the
form \eqref{2.20} were introduced by   Otelbaev (see
\cite{9}).\end{remark}

\begin{lem} \cite [\S6   below]{4}  \label{lem2.17}
Let $x\in\mathbb R,$ $t\in[x-d(x),x+d(x)]$
$(t\in[x-s(x),x+s(x)])$. Then the following inequalities hold:
\begin{equation}\label{2.25}
\alpha^{-1}v(x)\le v(t)\le\alpha v(x), \qquad \alpha^{-1}u(x)\le
u(t)\le\alpha u(x),
\end{equation}
\begin{equation}\label{2.26}
\alpha^{-1}\rho(x)\le \rho(t)\le\alpha \rho(x), \qquad
(4\alpha)^{-1}h(x)\le h(t)\le4\alpha h(x).
\end{equation}

\begin{equation}\label{2.27}
\left(\begin{array}{cc}
 c^{-1}v(x)\le v(t)\le cv(x),\quad c^{-1}u(x)\le u(t)\le
 cu(x)
  \\  \\
 c^{-1}\rho(x)\le\rho(t)\le c\rho(x)
\end{array}\right)
.
\end{equation}
 Here $\alpha=\exp(2).$
    \end{lem}

    \begin{thm}  \cite{2}  \label{thm2.18}
    Suppose that conditions \eqref{1.2} and \eqref{2.1} hold and
    $p\in(1,\infty).$ Then equation \eqref{1.1} is correctly
    solvable in $L_p$ if and only if the Green operator $G:L_p\to
    L_p $ is bounded. In the latter case, for every function $f\in
    L_p$ the solution $y\in L_p$ of \eqref{1.1} is of the form
    $y=Gf.$ In particular, $\mathcal L_p^{-1}=G.$ Here (see
    \eqref{2.8}):
    \begin{equation}
\label{2.29} (Gf)(x)\doe\int_{-\infty}^\infty G(x,t)f(t)dt,\quad
x\in\mathbb R,\quad f\in L_p.
\end{equation}
\end{thm}

\begin{remark}\label{rem2.19} If $r^{-1}\notin L_1(-\infty,0)$ and
$r^{-1}\notin L_1(0,\infty)$, then condition  \eqref{2.1} and, a
fortiori, \eqref{1.3} are necessary for correct solvability of
equation \eqref{1.1} in $L_p,$ $p\in(1,\infty)$ (see \cite{2}).
\end{remark}

\begin{lem}  \cite{2}  \label{lem2.20}
Suppose that conditions \eqref{1.2} and \eqref{2.1} hold and
$p\in(1,\infty).$ Consider the integral operators
\begin{equation}\label{2.30}
(G_1f)(x)=u(x)\int_{-\infty}^x v(t)f(t)dt,\qquad x\in\mathbb R,
\end{equation}
\begin{equation}\label{2.31}
(G_2f)(x)=v(x)\int_x^\infty u(t)f(t)dt,\qquad x\in\mathbb R.
\end{equation}
We have the relations
\begin{equation}\label{2.32}
G=G_1+G_2,\end{equation}
\begin{equation}\label{2.33}
 \frac{\|G_1\|_{p\to p}+\|G_2\|_{p\to p}}{2}\le\|G\|_{p\to p}\le \|G_1
\|_{p\to p}+\|G_2\|_{p\to p}.\end{equation}
\end{lem}

\begin{thm}  \cite{2}  \label{thm2.21}
Equation \eqref{1.1} is correctly solvable in $L_p$,
$p\in(1,\infty)$ if and only if $B<\infty.$ Here
\begin{equation}\label{2.34}
B\doe\sup_{x\in\mathbb R} h(x)d(x).
\end{equation}
Moreover, the following relations hold:
\begin{equation}\label{2.35}
\|G\|_{p\to p}\asymp\|G_1\|_{p\to p}\asymp\|G_2\|_{p\to p}\asymp
B.\end{equation}
\end{thm}

\begin{thm}\label{thm2.22}
Let \eqref{1.2} and \eqref{2.1} be satisfied. Then equation
\eqref{1.1} is correctly solvable in $L_p,$ $p\in(1,\infty)$ if
and only if $S<\infty.$ Here
\begin{equation}\label{2.36}
S\doe \sup_{s\in\mathbb R}(\rho(x)s(x)).
\end{equation}
\end{thm}

\begin{remark}\label{rem2.23}
Theorems \ref{thm2.21} and \ref{thm2.22} are proved in the same
way because the properties of the functions $d$ and $s,$\ $\rho$
and $h$ are quite analogous (see above). Moreover, the proof of
\thmref{thm2.22} is even simpler compared to \thmref{thm2.21}
because there is no need to apply estimates \eqref{2.14}. In
particular, for this reason, in \thmref{thm2.22} instead of
condition \eqref{1.3} of \thmref{thm2.21} there appears a weaker
condition \eqref{2.1}. Thus, since the proof of \thmref{thm2.22}
is reduced to the repetition of the argument from \cite{2}, we do
not present it here.
\end{remark}

\begin{thm}  \cite{2, 12}  \label{thm2.24} Suppose that the
conditions \eqref{1.2} and $r\equiv 1$ hold. Then equation
\eqref{1.1}is correctly solvable in $L_p,$ $p\in[1,\infty)$ if and
only if there exists $a>0$ such that $m(a)>0$. Here
$$m(a)=\inf_{x\in\mathbb R}\int_{x-a}^{x+a}q(t)dt.$$
\end{thm}

\begin{thm}  \cite{2, 4}  \label{thm2.25} For every
$p\in(1,\infty)$ equation \eqref{1.1} is correctly solvable in
$L_p$ if $A>0.$ Here
\begin{equation}\label{2.37}
\mathcal A= \inf_{x\in\mathbb R}\mathcal A(x),\qquad \mathcal
A(x)= \frac{1}{2d(x)}\int_{x-d(x)}^{x+d(x)}q(t)dt.\end{equation}
\end{thm}

\begin{remark}\label{rem2.26}
In contrast to the condition $B<\infty,$ the meaning of the
requirement $\mathcal A>0$ is quite obvious: some special Steklov
average of the function $q$ must be separated from zero uniformly
on the whole axis (see \cite{4}). Moreover, the requirement $A>0$
can be viewed as a weakening of the simplest condition
$\inf\limits_{x\in\mathbb R}q(x)>0$ guaranteeing correct
solvability of \eqref{1.1} in $L_p$, $p\in[1,\infty)$ (see
\cite{4,7}. We continue this comment in the  next assertion
(\thmref{thm2.28}) by defining a meaningful class of equations
\eqref{1.1} (see \cite{10}) in which the requirement $B<\infty$ is
equivalent to a condition of the form $\mathcal A>0.$ Towards this
end, we need a new auxiliary function.
\end{remark}

\begin{lem} \cite{10, 13}  \label{lem2.27} Let $\varphi(x)\asymp
\psi(x),$ $x\in\mathbb R.$ For a given $x\in\mathbb R$ consider
the equation in $\mu\ge0:$
\begin{equation}\label{2.38}
\int_{x-\mu}^{x+\mu}q(t)h(t)dt=1.
\end{equation}
Equation \eqref{2.38} has at least one positive finite solution.
Let
\begin{equation}\label{2.39}
\mu(x)=\inf_{\mu\ge0}\left\{\mu:\int_{x-\mu}^{x+\mu}q(t)h(t)dt=1\right\}.
\end{equation}
The function $\mu(x)$ is continuous for $x\in\mathbb R$, and, in
addition,
\begin{equation}\label{2.40}
\lim_{x\to-\infty}(x+\mu(x))=-\infty, \qquad
\lim_{x\to\infty}(x-\mu(x))=\infty.\end{equation}
\end{lem}

\begin{thm}  \cite{10}  \label{thm2.28}
Let $\varphi(x)\asymp \psi(x),$ $x\in\mathbb R.$ Then $B<\infty$
if and only if $\tilde{\mathcal A}>0.$ Here
\begin{equation}\label{2.41}
\tilde{\mathcal A}=\inf_{x\in\mathbb R}\tilde{\mathcal
A}(x),\qquad \tilde{\mathcal A}(x)=\frac{1}{2\mu(x)}
\int_{x-\mu(x)}^{x+\mu(x)}q(t)dt.\end{equation}
\end{thm}

\begin{remark}\label{rem2.29}
To apply \thmref{thm2.21} to concrete equations, one has to know
the auxiliary functions $h$ and $d.$ Usually it is not possible to
express these functions through the original coefficients $r$ and
$q$ of equation \eqref{1.1}. However, it is easy to see that when
studying the value of $B,$ one can replace in an equivalent way
the functions $h$ and $d$ with their sharp by order two-sided
estimates. In most cases, such inequalities can be obtained using
standard tools of local analysis (see, e.g., \cite{4} and a
detailed exposition in \cite{10}; one example of obtaining such
estimates is given in \S6 below). It is clear that in concrete
cases of the question on I)--II), it is particularly convenient to
use criteria which either do not use the functions $h$ and $d$ at
all, or use, say, only the function $h.$ Such assertions are
contained in the following theorem.
\end{remark}

\begin{thm}  \cite{2}  \label{thm2.30}
Suppose that conditions \eqref{1.2}--\eqref{1.3} hold. Then we
have the following assertions:

{\rm A)} Equation \eqref{1.1} is correctly solvable in $L_p,$
$p\in(1,\infty)$ if any of the following conditions holds:
\begin{alignat}{4}
& 1)\quad && B_1<\infty, \quad && B_1=\sup_{x\in\mathbb R}B_1(x),
\quad && B_1(x)=r(x)h^2(x),\qquad\qquad\qquad\qquad\label{2.42}\\
& 2)\quad && B_2<\infty, \quad && B_2=\sup_{x\in\mathbb R}B_2(x),
\quad && B_2(x)=r(x)\varphi(x)\psi(x),\qquad\qquad\qquad\qquad\label{2.43}\\
& 3)\quad && B_3<\infty, \quad && B_3=\sup_{x\in\mathbb R}B_3(x),
\quad && B_3(x)=
h(x)\cdot|x|,\qquad\qquad\qquad\qquad\label{2.44}\end{alignat}

{\rm B)} Suppose that in addition to \eqref{1.2} and \eqref{1.3}
the following conditions hold:
\begin{equation}\label{2.45}
r^{-1}\in L_1,\qquad q\notin L_1(-\infty,0),\qquad q\notin
L_1(0,\infty).\end{equation} Then equation \eqref{1.1} is
correctly solvable  in $L_p,$ $p\in(1,\infty)$ if $\theta<\infty.$
Here $\theta=\sup_{x\in\mathbb R}\theta(x),$
\begin{equation}\label{2.46}
\theta(x)=|x|\left(\int_{-\infty}^x\frac{dt}{r(t)}\right)\cdot
\left(\int_x^\infty \frac{dt}{r(t)}\right).
\end{equation}
\end{thm}

We also need the following known facts.

\begin{thm}  \cite[Ch.IV, \S8, Theorem 20]{14}  \label{thm2.31}
Let $p\in(1,\infty).$ The set $\mathcal K\in L_p$ is precompact if
and only if the following conditions hold:
\begin{alignat}{2}
& 1)\quad && \sup_{f\in\mathcal
K}\|f\|_p<\infty,\qquad\qquad\qquad\qquad\qquad\qquad\qquad\label{2.47}\\
& 2)\quad && \lim_{\delta\to0} \sup_{f\in\mathcal
K}\sup_{|t|<\delta}\|f(\cdot+t)-f(\cdot)\|_p=0
,\qquad\qquad\qquad\qquad\qquad\qquad\qquad\label{2.48}\\
& 3)\quad && \lim_{x\to\infty} \sup_{f\in\mathcal K}\int_{|x|\ge
N}|f(x)|^pdx=0.\qquad\qquad\qquad\qquad\qquad\qquad\qquad\label{2.49}\end{alignat}
\end{thm}

Let $\mu,\theta$ be almost everywhere finite measurable positive
functions defined in the interval $(a,b),-\infty\le a<b\le\infty.$

We introduce the integral operators
\begin{equation}\label{2.50}
(Kf)(x)=\mu(x)\int_x^b\theta(t)f(t)dt,\qquad x\in(a,b),
\end{equation}
\begin{equation}\label{2.51}
(\tilde Kf)(x)=\mu(x)\int_a^x\theta(t)f(t)dt,\qquad x\in(a,b).
\end{equation}

\begin{thm}  \cite{15} \cite[Ch.1, \S1.3]{16}  \label{thm2.32}
For $p\in(1,\infty)$ the operator $K:L_p(a,b)\to L_p(a,b)$ is
bounded if and only if $H_p(a,b)<\infty.$ Here
$H_p(a,b)=\sup_{x\in(a,b)}H_p(x,a,b),$
\begin{equation}\label{2.52}
H_p(x,a,b)=\left[\int_a^x\mu(t)^pdt\right]^{1/p}\cdot\left[\int_x^b\theta(t)^{p'}dt\right]^{1/p'},\quad
p'=\frac{p}{p-1}.\end{equation}
\end{thm}

In addition, the following inequalities hold:
\begin{equation}\label{2.53}
H_p(a,b)\le\|K\|_{L_p(a,b)\to
L_p(a,b)}\le(p)^{1/p}(p')^{1/p'}H_p(a,b),
\end{equation}

\begin{thm}  \cite{15} \cite[Ch.1, \S1.3]{16}  \label{thm2.33}
For $p\in(1,\infty)$ the operator $\tilde K:L_p(a,b)\to L_p(a,b)$
is bounded if and only if $\tilde H_p(a,b)<\infty.$ Here $\tilde
H_p(a,b)=\sup_{x\in(a,b)}\tilde H_p(x,a,b),$ and
\begin{equation}\label{2.54}
\tilde
H_p(x,a,b)=\left[\int_a^x\theta(t)^{p'}dt\right]^{1/p'}\cdot\left[\int_x^b\mu(t)^{p}dt\right]^{1/p},\quad
p'=\frac{p}{p-1}.\end{equation}
\end{thm}

In addition, the following inequalities hold:
\begin{equation}\label{2.55}
\tilde H_p(a,b)\le\|K\|_{L_p(a,b)\to
L_p(a,b)}\le(p)^{1/p'}(p')^{1/p'}\tilde H_p(a,b).
\end{equation}

Note that some assertions (mainly of a technical nature) will be
given in \S4--\S5 in the course of the exposition.

\section{Main Results}

Recall that, if conditions I)--II) hold, then $\mathcal
L_p^{-1}=G,$ $p\in(1,\infty)$ (see \thmref{thm2.18}). Therefore,
in the sequel in the statements of the theorems, we write the
operator $G$ instead of the operator $\mathcal L_p^{-1}.$

Our main result is the following theorem.

\begin{thm}\label{thm3.1} Let $p\in(1,\infty)$, and suppose that
equation \eqref{1.1} is correctly solvable in $L_p.$ Then the
operator $G:L_p\to L_p$ is compact if and only if
\begin{equation}\label{3.1}
\lim_{|x|\to\infty}h(x)d(x)=0.
\end{equation}
\end{thm}

\begin{thm}\label{thm3.2}
Suppose that conditions \eqref{1.2} and \eqref{2.1} hold,
$p\in(1,\infty),$ and equation \eqref{1.1} is correctly solvable
in $L_p.$ Thus the operator $G: L_p\to L_p$ is compact if and only
if
\begin{equation}\label{3.2}
\lim_{|x|\to\infty} \rho(x)s(x)=0.\end{equation}
\end{thm}

\begin{remark}\label{rem3.3}
Theorems \ref{thm3.1} and \ref{thm3.2} are related to one another
in the same way as Theorems \ref{thm2.21} and \ref{thm2.22} (see
Remark \ref{rem2.23}). Therefore, we do not present a proof of
\thmref{thm3.2}.
\end{remark}

\begin{thm}\label{thm3.4}
Suppose that condition  \eqref{3.1} holds. Then the operator $G:
L_2\to L_2$ is compact, self-adjoint, and positive. Its maximal
and eigenvalue $\lambda$ satisfies the estimates (see
\eqref{2.34}):
\begin{equation}\label{3.3}
c^{-1}B\le \lambda\le cB.
\end{equation}
\end{thm}

\begin{remark}\label{rem3.5}
Theorems \ref{thm3.1} and \ref{thm3.4} were obtained in \cite{13}
under an additional requirement $\varphi\asymp\psi(x),$
$x\in\mathbb R.$ The meaning of condition \eqref{3.1} can be
clarified ``in terms of the coefficients" of equation \eqref{1.1}
in the same way as is done in Remark \ref{rem2.26} for the
interpretation of the condition $B<\infty.$ In particular, in
order to expand on \thmref{thm3.1}, we state the following
theorem.
\end{remark}

\begin{thm}\label{thm3.6} \cite{13}
Let $\varphi\asymp\psi(x),$\ $x\in\mathbb R,$ $\ p\in(1,\infty),$
and suppose that equation \eqref{1.1} is correctly solvable in
$L_p.$ Then the operator $G: L_p\to L_p$ is compact if and only if
\begin{equation}\label{3.4}
\lim_{|x|\to\infty}\tilde{\mathcal A}(x)=\infty
\end{equation}
(see \eqref{2.41}).\end{thm}
 Thus, if $\varphi(x)\asymp\psi(x),$
$x\in\mathbb R$, requirement \eqref{3.1} means that some special
Steklov average value of the function $q$ must tend to infinity at
infinity.

We can now present several consequences of \thmref{thm3.1}. Their
significance consists in the fact that they allow us to clarify
the question on I)--III)) either not using at all the functions
$k$ and $d$, or with the help of only $h$ (see Remark
\ref{rem2.29}.

\begin{cor}\label{cor3.7}
Let $p\in(1,\infty)$ and $\mathcal A>0$ (see \eqref{2.37}). Then
the operator $G: L_p\to L_p$ is compact if $\mathcal
A(x)\to\infty$ as $|x|\to\infty.$ \end{cor}

\begin{cor}\label{cor3.8}
Let $p\in(1,\infty)$ and  $q(x)\to\infty$ as $|x|\to\infty.$ Then
the operator $G:L_p\to L_p$ is compact.
\end{cor}

\begin{cor}\label{cor3.9} \cite{17}
Suppose that conditions \eqref{1.2} hold, $r(x)\equiv1,$
$x\in\mathbb R,$ and $m(a_0)>0$ for some $a_0\in(0,\infty)$ (see
\thmref{thm2.24}). Then the operator $G:L_p\to L_p$ is compact if
and only if the Molchanov condition (see \cite{2}) holds:
\begin{equation}\label{3.5}
\lim_{|x|\to\infty}\int_{x-a}^{x+a}q(t)dt=\infty,\qquad \forall
a\in(0,\infty).
\end{equation}
\end{cor}

\begin{cor}\label{cor3.10}
Let $p\in(1,\infty)$. Then assertions I)--III) hold if and only if
any of the following conditions is satisfied:
 \begin{alignat}{2}
&1)\ && B_1<\infty \ \text{(see \eqref{2.42})},\  r(x)h^2(x)\to 0
\
 \text{as}\
 |x|\to\infty\label{3.6}\qquad
\qquad\qquad\qquad\\
&2)\ && B_2<\infty \ \text{(see \eqref{2.43})},\
 r(x)\varphi(x)\psi(x)\to 0 \
 \text{as}\
 |x|\to\infty\label{3.7}\\
 &3)\ && B_3<\infty \ \text{(see \eqref{2.44})},\  h(x)\cdot|x|\to
 0  \
 \text{as}\
 |x|\to\infty\label{3.8}
\end{alignat}
\end{cor}

\begin{cor}\label{cor3.11}
Denote
\begin{equation}\label{3.9}
r_0=\sup_{x\in\mathbb R} r(x),\qquad h_0=\sup_{x\in\mathbb R}
h(x).
\end{equation}
Let $r_0<\infty.$ Then \eqref{1.1} is correctly solvable in $L_p,$
$p\in(1,\infty)$ if $h_0<\infty.$ In addition, the operator $G:
L_p\to L_p,$ $p\in(1,\infty)$ is compact if $h(x)\to 0$ as
$|x|\to\infty.$
\end{cor}

\begin{remark}\label{rem3.12}
Note that the requirement
\begin{equation}\label{3.10}
q(x)\to \infty\qquad\text{as}\qquad |x|\to\infty
\end{equation}
is so strong that the answer to the question on I)--III) is no
dependent on the behaviour (within the framework of \eqref{1.2})
of the function $r.$ In this connection, look at the opposite
situation and find out what requirements on the function $r$ is
the positive solution of the behaviour (within a certain
framework) of the function $q$. See Theorems \ref{thm3.13} and
\ref{thm3.14} below for possible answers to these questions.

We emphasize that these assertions have been obtained from
Theorems \ref{thm2.22} and \ref{3.2} where \eqref{1.3} is not
used. Therefore, n Theorems \ref{thm3.13} and \ref{thm3.14}
requirements on the function $q$ are weakened to conditions
\eqref{1.2} and \eqref{2.1}.
\end{remark}

\begin{thm}\label{thm3.13}
Suppose that together with \eqref{1.2} condition \eqref{2.1} holds
and $\theta<\infty $ (see \thmref{thm2.30}). Then equation
\eqref{1.1} is correctly solvable in $L_p,$ $p\in(1,\infty).$ In
addition, the operator $G: L_p\to L_p$, $p\in(1,\infty)$ is
compact if $\theta(x)\to0$ as $|x|\to\infty$ (see \eqref{2.46}).
\end{thm}

\begin{thm}\label{thm3.14}
Suppose that  conditions \eqref{1.2}, \eqref{2.1} hold  and
$\nu<\infty $. Here $\nu=\sup\limits_{x\in\mathbb R}\nu(x),$
\begin{equation}\label{3.11}
\nu(x)=r(x)\left(\int_{-\infty}^x\frac{dt}{r(t)}\right)^2\cdot
\left(\int_x^\infty\frac{dt}{r(t)}\right)^2,\quad x\in\mathbb
R.\end{equation}
 Then equation \eqref{1.1} is correctly solvable
in $L_p,$ $p\in(1,\infty).$ If, in addition, $\nu(x)\to0$ as
$|x|\to\infty,$ then the operator $G:L_p\to L_p,$ $p\in(1,\infty)$
is compact.
\end{thm}

\section{Proofs}

\renewcommand{\qedsymbol}{}
\begin{proof}[Proof of \thmref{thm3.1}]  Necessity.

Let us check \eqref{3.1} as $x \to\infty.$ (The case $x\to-\infty$
is treated in a similar way.) Let $\{\Delta_n\}_{n\in\mathbb N'}$
be an $\mathbb R(0,d)$-covering of $\mathbb R,$\
$F=\{f_n(t)\}_{n\in\mathbb N'}$ and
$$f_n(t)=\begin{cases} d(x_n)^{-1/p}&\quad\text{if}\
 t\in\Delta_n \\
&\qquad\qquad\qquad, \quad n\in \mathbb N'\\
0&\quad\text{if}\
 t\notin\Delta_n \end{cases}$$
Then $\|f_n\|_p^p=2,$ $n\in\mathbb N'$ and the set
$\{Gf_n\}_{n\in\mathbb N'}$ is precompact in $L_p.$ Let
$x\in\Delta_n,$ $n\in\mathbb N'.$ In the following relations we
apply \eqref{2.14} and \eqref{2.25}--\eqref{2.26}:
\begin{align}
(Gf_n)(x)&=u(x)\int_{\Delta_n^-}^x
v(t)f_n(t)dt+v(x)\int_x^{\Delta_n^+}u(t)f_n(t)dt\nonumber\\
&=\frac{u(x)}{u(x_n)}\rho(x_n)\int_{\Delta_n^-}^x\frac{v(t)}{v(x_n)}\cdot\frac{dt}{d(x_n)^{1/p}}
+\frac{v(t)}{v(x_n)}\cdot\rho(x_n)\int_x^{\Delta_n^+}\frac{u(t)}{u(x_n)}\frac{dt}{d(x_n)^{1/p}}
\nonumber\\
&\ge c^{-1}\rho(x_n)d(x_n)^{1/p'}\ge
c^{-1}h(x_n)d(x_n)^{1/p'},\quad n\in \mathbb N'.\label{4.1}
\end{align}

By \thmref{thm2.31}, for a given $\varepsilon>0$ there exists
$\mathbb N(\varepsilon)\gg1$ such that
$$\sup_{f_k\in F}\int_{|x|\ge\mathbb
N(\varepsilon)}|(Gf_k)(t)|^pdt\le\varepsilon.$$ From the
properties of an $\mathbb R(0,d)$-covering of $\mathbb R,$ it
follows that there exists $n_0=n_0(\varepsilon)\in\mathbb
N=\{1,2,3,\dots\}$ such that $\mathbb
N(\varepsilon)\in\Delta_{n_0}.$ Set
$n_1=n_1(\varepsilon)=n_0(\varepsilon)+1.$ Since $\mathbb
N(\varepsilon)\le\Delta_{n_1}^-,$ we have
\begin{equation}\label{4.2}
\sup_{f_k\in
F}\int_{\Delta_{n_1}^-}^\infty|(Gf_k)(t)|^pdt\le\varepsilon.\end{equation}
Let $k\ge n_1(\varepsilon)$. Then from \eqref{4.1}--\eqref{4.2},
it follows that
$$\varepsilon\ge\int_{\Delta_{n_1}^-}^\infty|(Gf_k)(t)|^pdt\ge\int_{\Delta_k^-}^{\Delta_k^+}
|(Gf_k)(t)|^pdt\ge c^{-1}(h(x_k)d(x_k))^p.$$ Therefore,
$\lim\limits_{k\to\infty}(h(x_k)d(x_k))=0.$ From the inequalities
$$0< h(x)d(x)\le ch(x_n)d(x_n),\qquad x\in\Delta_n,\quad
n\in\mathbb N',$$ that follows from \eqref{2.26} and \eqref{2.20}
(for $\varepsilon=0$), we note get \eqref{3.1}.

 \end{proof}

\renewcommand{\qedsymbol}{\openbox}

\begin{proof}[Proof of \thmref{thm3.1}]  Sufficiency.
Assume that the hypotheses of the theorem are satisfied. Then by
\thmref{thm2.18} the operator $G: L_p\to L_p$ is bounded, and by
\lemref{lem2.20} so are the operators $G_1: L_p\to L_p$ and $G_2:
L_p\to L_p.$ Clearly, if $G_1$ and $G_2$ are compact, then so is
$G$ (see \eqref{2.32}). Furthermore, compactness of $G_1$ and
$G_2$ is checked in the same way, and therefore below we only
consider $G_2.$

Let $F=\{f\in L_p:\|f\|_p\le 1\}.$ Compactness of $G_2: L_p\to
L_p$ will be established as soon as we check that the set
$W=\{g\in L_p:g=G_2f,\ f\in F\}$ is precompact in $L_p.$ Below we
show that the set $W$ satisfies conditions 1), 2), and 3) of
\thmref{thm2.31} and thus proves \thmref{thm3.1}.

\textit{Verification of condition 1)}. The above arguments
(together with the definition of the set $F$ and Theorems
\ref{thm2.18} and \ref{thm2.21}) imply the inequality  1),
\begin{align*}
\sup_{g\in W}\|g\|_p&=\sup_{f\in F}\|G_2f\|_p\le \|G_2\|_{p\to
p}\cdot\|f\|_p\\
&\le \|G_2\|_{p\to p}\le cB<\infty\quad \Rightarrow
1).\end{align*}

\textit{Verification of condition 3)}. We need some auxiliary
assertions.

\begin{lem}\label{lem4.1} \cite{2} Let $x\in\mathbb R$ and let
$\{\Delta_n\}_{n\in\mathbb N'}$ be an $\mathbb R(x,d)$-covering of
$\mathbb R.$ Then \begin{equation}
\begin{aligned}\label{4.3}
\int_{\Delta_n^+}^{\Delta_{-1}^+}\frac{d\xi}{r(\xi)h(\xi)}&=|n|-1,\quad
\text{if}\quad n\le -1\\
\int_{\Delta_1^-}^{\Delta_n^-}\frac{d\xi}{r(\xi)h(\xi)}&=n-1,\quad\
\ \text{if}\quad n\ge  1.\end{aligned} \end{equation}
\end{lem}

\begin{lem}\label{lem4.2}
Let $B<\infty$ (see \eqref{2.34}), $x\in\mathbb R,$
$p\in(1,\infty)$ and
\begin{equation}\label{4.4}
\theta_p(x)=\left[\int_{-\infty}^x
v(t)^pdt\right]^{1/p}\cdot\left[\int_x^\infty
u(\xi)^{p'}d\xi\right]^{1/p'},\quad p'=\frac{p}{p-1}.
\end{equation}
Then we have the inequalities \begin{equation}\label{4.5}
c^{-1}h(x)d(x)\le\theta_p(x)\le\begin{cases}
B^{1/p}\sup\limits_{t\ge x}(h(t)d(t))^{1/p'}, \quad \text{if}\
x\ge 0\\
B^{1/p'}\sup\limits_{t\le x}(h(t)d(t))^{1/p}, \quad \text{if}\
x\le 0\end{cases}\end{equation}
\end{lem}

\begin{proof}
Let $p\in(1,2],$ $\gamma\in(0,1]$ (the number $\gamma$ will be
chosen later). Now we apply Theorems \ref{thm2.1} and
\ref{thm2.3}:
\begin{align}
\theta_p(x)&\le\left[\int_{-\infty}^x v(t)^pdt\right]^{1/p}\cdot
u(x)^\gamma\cdot\left[\int_x^\infty
u(\xi)^{(1-\gamma)p'}d\xi\right]^{1/p'}\nonumber\\
&\le\left[\int_{-\infty}^x\rho(t)^{\gamma p}\cdot
v(t)^{(1-\gamma)p}dt\right]^{1/p}\cdot\left[\int_x^\infty
u(\xi)^{(1-\gamma)p'}d\xi\right]^{1/p'}\nonumber\\
&=\left[\int_{-\infty}^x\rho(t)^{\frac{1+\gamma}{2}p}\exp\left(-\frac{1-\gamma}{2}p\int_t^x
\frac{ds}{r(s)\rho(s)}\right)\cdot\exp\left(\frac{1-\gamma}{2}p\int_{x_0}^x\frac{ds}{r(s)\rho(s)}
\right)dt\right]^{1/p}\nonumber\\
&\quad \cdot
\left[\int_x^\infty\rho(\xi)^{\frac{1-\gamma}{2}p'}\exp\left(-\frac{1-\gamma}{2}p'
\int_x^\xi\frac{ds}{r(s)\rho(s)}\right)\exp\left(-\frac{1-\gamma}{2}p'\int_{x_0}^x\frac{ds}
{r(s)\rho(s)}\right)d\xi\right]^{1/p'}\nonumber\\
&=\left[\int_{-\infty}^x\rho(t)^{\frac{1+\gamma}{2}p}\exp\left(-\frac{1-\gamma}{2}p\int_t^x
\frac{ds}{r(s)\rho(s)}\right)dt\right]^{1/p}\nonumber\\
&\quad\cdot\left[\int_x^\infty\rho(\xi)
^{\frac{1-\gamma}{2}p'}\exp
\left(-\frac{1-\gamma}{2}p'\int_x^\xi\frac{ds}{r(s)\rho(s)}\right)d\xi\right]^{1/p'}.\label{4.6}
\end{align}

Let $\gamma_1$ be the solution of the equation
$$\frac{1+\gamma}{2}p=\frac{1-\gamma}{2}p'\quad\Rightarrow\quad\gamma:=\gamma_1=\frac{p'-p}{p'+p}.$$
  For $\gamma=\gamma_1$
inequality \eqref{4.6} takes the form
\begin{align}
\theta_p(x)&\le\left[\int_{-\infty}^x\rho(t)\exp\left(-(p-1)\int_t^x\frac{ds}{r(s)\rho(s)}
\right)
dt\right]^{1/p}\nonumber\\
 &\quad\cdot
 \left[\int_x^\infty\rho(\xi)\exp
\left(-\int_x^\xi\frac{ds}{r(s)\rho(s)}\right)d\xi\right]
 ^{1/p'}:=(J_1(x))^{1/p}\cdot(J_2(x))^{1/p'}.\label{4.7}
\end{align}

  Let us estimate $J_1(x)$ and $J_2(x)$. We only consider the case
  $x\ge0$ because the case $x\le0$ is treated in a similar way.
  Below we use the properties of an $\mathbb R(x,d)$-covering of
  $\mathbb R,$ inequalities \eqref{2.26} and \eqref{2.14}, and
  equalities \eqref{4.3}:
  \begin{align}
  J_1(x)&=\int_{-\infty}^x\rho(t)\exp\left(-(p-1)\int_t^x\frac{ds}{r(s)\rho(s)}\right)dt\nonumber\\
  &=\sum_{h=-\infty}^{-1}\int_{\Delta_n}\rho(t)\exp\left(-(p-1)\int_t^x\frac{ds}{r(s)\rho(s)}\right)dt\nonumber\\
 &\le c\sum_{h=-\infty}^{-1}
  h(x_n)d(x_n)\exp\left(-\frac{p-1}{2}\int_{\Delta_{n^+}}^{\Delta_{-1}^+}
  \frac{ds}{r(s)h(s)}\right)\nonumber\\
   &\le
  cB\sum_{n=-\infty}^{-1}\exp\left(-\frac{p-1}{2}(|n|-1)\right)=cB,\label{4.8}
\end{align}

\begin{align}
  J_2(x)&=\int_x^\infty\rho(\xi)\exp\left(-\int_x^\xi\frac{ds}{r(s)\rho(s)}\right)d\xi
  =\sum_{n=1}^\infty\int_{\Delta_n}\rho(\xi)\exp\left(-\int_x^\xi\frac{ds}{r(s)\rho(s)}\right)d\xi\nonumber\\
  &\le c\sum_{n=1}^\infty
  h(x_n)d(x_n)\exp\left(-\frac{1}{2}\int_{\Delta_1^-}^{\Delta_n^-}\frac{ds}{r(s)\rho(s)}\right)\nonumber\\
  &\le c\sup_{t\ge
  x}(h(t)d(t))\sum_{n=1}^\infty\exp\left(-\frac{n-1}{2}\right)=c
  \sup_{t\ge x}(h(t)d(t)).\label{4.9}
  \end{align}

  Thus, for $p\in(1,2]$, the upper estimate in \eqref{4.5} follows
  from \eqref{4.8}--\eqref{4.9}. Let $p\in(2,\infty),$
  $\gamma\in(0,1]$ (the number $\gamma$ will be chosen later).
  Now we apply Theorems \ref{thm2.1} and \ref{thm2.3}:
\begin{align}
\theta_p(x)&=\left[\int_{-\infty}^x
v(t)^pdt\right]^{1/p}\cdot\left[\int_x^\infty
u(\xi)^{p'}d\xi\right]^{1/p'}\nonumber\\
&\le\left[\int_{-\infty}^x v(t)^{(1-\gamma)p}dt\right]^{1/p}\cdot
v(x)^\gamma\cdot\left[\int_x^\infty
u(\xi)^{p'}d\xi\right]^{1/p'}\nonumber\\
&\le\left[\int_{-\infty}^x
v(t)^{(1-\gamma)p}dt\right]^{1/p}\cdot\left[\int_x^\infty\rho(\xi)^{\gamma
p'}\cdot u(\xi)^{(1-\gamma)p'}d\xi\right]^{1/p'}\nonumber\\
&=\left[\int_{-\infty}^x\rho(t)^{\frac{1-\gamma}{2}p}\cdot\exp\left(-\frac{1-\gamma}{2}p\int_t^x
\frac{ds}{r(s)\rho(s)}\right)\cdot\exp\left(\frac{1-\gamma}{2}p\int_{x_0}^x\frac{ds}{r(s)\rho(s)}\right)
dt\right]^{1/p}\nonumber\\
&\quad
\cdot\left[\int_x^\infty\rho(\xi)^{\frac{1+\gamma}{2}p'}\cdot\exp\left(-\frac{1-\gamma}{2}p'\int_x^\xi
\frac{ds}{r(s)\rho(s)}\right)\cdot\exp\left(-\frac{1-\gamma}{2}p'\int_{x_0}^x\frac{ds}{r(s)\rho(s)}\right)
d\xi\right]^{1/p'}\nonumber\\
&=\left[\int_{-\infty}^x\rho(t)^{\frac{1-\gamma}{2}p}\exp\left(-\frac{1-\gamma}{2}p\int_t^x\frac{ds}
{r(s)\rho(s)}\right)dt\right]^{1/p}\nonumber\\
&\quad
\cdot\left[\int_{x}^\infty\rho(\xi)^{\frac{1+\gamma}{2}p'}\exp\left(-\frac{1-\gamma}{2}p\int_x^\xi\frac{ds}
{r(s)\rho(s)}\right)d\xi\right]^{1/p'}.\label{4.10}
 \end{align}

Let now $\gamma$ be the solution $\gamma_2$ of the equation
$$\frac{1-\gamma}{2}p=\frac{1+\gamma}{2}p'\quad\Rightarrow\quad
\gamma:=\gamma_2=\frac{p-p'}{p+p'}.$$ For $\gamma=\gamma_2$
inequality \eqref{4.10} takes the form
\begin{align}
\theta_p(x)&\le\left[\int_{-\infty}^x\rho(t)\exp\left(-\int_t^x\frac{ds}{r(s)\rho(s)}\right)dt\right]
^{1/p}
\nonumber\\
 &\quad\cdot
 \left[\int_x^\infty\rho(\xi)\exp\left(-(p'-1)\int_x^\xi\frac{ds}{r(s)\rho(s)}\right)d\xi\right]^{1/p'}.
 \label{4.11}
 \end{align}

 That \eqref{4.11} implies the upper estimate in \eqref{4.5} can
 be proved similarly to the proof of the same estimate from
 \eqref{4.7}, and therefore we omit the proof. It remains to
 obtain the lower estimate in \eqref{4.5}. The following inequality
follows from \eqref{2.14} and \eqref{2.26}:
\begin{align*}
\theta_p(x)&\ge\left[\int_{x-d(x)}^xv(t)^pdt\right]^{1/p}\cdot\left[\int_x^{x+d(x)}u(t)^{p'}dt\right]
^{1/p'} \\
&\ge c^{-1}v(x)d(x)^{1/p}\cdot
c^{-1}u(x)d(x)^{1/p'}=c^{-1}\rho(x)d(x)\ge c^{-1}h(x)d(x).
\end{align*}
\end{proof}

\begin{cor}\label{cor4.3}
Let $p\in(1,\infty)$ and $B<\infty$ (see \eqref{2.34}). Then
$\theta_p(x)\to0$ as $|x|\to\infty$ if and only if condition
\eqref{3.1} holds.\end{cor}

\begin{proof}
This is an immediate consequence of \eqref{4.5}.
\end{proof}

\begin{cor}\label{cor4.4}
Let $p\in(1,\infty)$ and $B<\infty$ (see \eqref{2.34}). Suppose
that condition \eqref{3.1} holds, $N\ge 1$ and
\begin{equation}\label{4.12}
\theta_p^{(+)}(x,N)=\left[\int_N^xv(t)^pdt\right]^{1/p}\cdot\left[\int_x^\infty
u(\xi)^{p'}d\xi\right]^{1/p'},\quad x\ge N,
\end{equation}
\begin{equation}\label{4.13}
\theta_p^{(-)}(x,N)=\left[\int_{-\infty}^xv(t)^pdt\right]^{1/p}\cdot\left[\int_x^{-N}
u(\xi)^{p'}d\xi\right]^{1/p'},\quad x\le -N.
\end{equation}
Then
\begin{equation}\label{4.14}
\theta_p^{(-)}(x,N)\to0,\qquad
\theta_p^{(+)}(x,N)\to0\quad\text{as}\quad  N\to\infty.
\end{equation}
\end{cor}

\begin{proof}
Now we use \eqref{4.5}:
\begin{align*}
0&<\theta_p^{(+)}(x,N)\le\sup_{x\ge
N}\theta_p^{(+)}(x,N)\le\sup_{x\ge N}\theta_p(x)\nonumber\\
&\le cB^{1/p}\sup_{t\ge N}(h(t)d(t))^{1/p'}\to
0\quad\text{as}\quad N\to\infty\quad\Rightarrow\quad \eqref{4.14}.
\end{align*}
The second relation of \eqref{4.14} can be checked in a similar
way.
\end{proof}

Let us now check 3). The following relations are obvious:
\begin{align*}
\sup_{g\in W}\int_{|x|\ge N}|g(t)|^pdt&=\sup_{f\in F}\int_{|x|\ge
N}|(G_2f)(x)|^pdx \\
&\le 2\sup_{f\in F}\max\left\{\int_{-\infty}^{-N}|(G_2f)(x)|^pdx,\
\int_N^\infty|(G_2f)(x)|^pdx\right\}.\end{align*} Denote
\begin{equation}\label{4.15}
T_1(N)=\sup_{f\in F}\int_{-\infty}^{-N}|(G_2f)(x)|^pdx,
\end{equation}
\begin{equation}\label{4.16}
T_2(N)=\sup_{f\in F}\int^{\infty}_{N}|(G_2f)(x)|^pdx.
\end{equation}
To prove 3), it is enough to verify that
\begin{equation}\label{4.17}
T_1(N)\to 0,\qquad T_2(N)\to0\qquad \text{as}\quad N\to\infty.
\end{equation}
Let us check \eqref{4.17} for $T_2(N).$ Now we use the definition
of the set $F,$ \thmref{thm2.32} and Corollary \ref{cor4.4}:
\begin{align*}
T_2(N)&=\sup_{f\in
F}\int_N^\infty|(G_2f)(x)|^pdx\le\|G_2\|^p_{L_p(N,\infty)\to
L_p(N,\infty)}\sup\|f\|_{L_p(N,\infty)}^p\\
&\le c(p)\sup_{x\ge N}\left[\left(\int_N^\infty
v(t)^pdt\right)^{1/p}\cdot\left(\int_x^\infty
u(\xi)^{p'}d\xi\right)^{1/p'}\right]^p\cdot\sup_{f\in F}\|f\|_p^p\\
&\le c(p)\sup_{x\ge N}[\theta_p^{(+)}(x,N)]^p\to 0\qquad
\text{as}\quad N\to\infty.
\end{align*}

Let us go to $T_1(N).$ First consider the value $(G_2f)(x)$ for
$x\le-N$ and $f\in F$:
\begin{align}
(G_2f)(x)&=v(x)\int_x^\infty
u(\xi)f(\xi)d\xi=v(x)\int_x^{-N}u(\xi)f(\xi)d\xi+v(x)\int_{-N}^\infty
u(\xi)f(\xi)d\xi\nonumber\\
&:=(\tilde P_Nf)(x)+(\hat P_Nf)(x).\label{4.18}
\end{align}
Here
\begin{equation}\label{4.19}
(\tilde P_Nf)(x)=v(x)\int_x^{-N}u(\xi)f(\xi)d\xi,\quad x\le
-N,\quad f\in F,
\end{equation}
\begin{equation}\label{4.20}
(\hat P_Nf)(x)=v(x)\int^\infty_{-N}u(\xi)f(\xi)d\xi,\quad x\le
-N,\quad f\in F.
\end{equation}

The following relations are obvious:
\begin{align}
T_1(N)&= \sup_{f\in
F}\int_{-\infty}^{-N}|(G_2f)(x)|^pdx=\sup_{f\in
F}\int_{-\infty}^{-N}|(\tilde P_Nf)(x)+(\hat
P_Nf)(x)|^pdx\nonumber\\
&\le 2^p\sup_{f\in F}\left[\int_{-\infty}^{-N}|(\tilde
P_Nf)(x)|^pdx+\int_{-\infty}^{-N}|(\hat
P_Nf)(x)|^pdx\right]\nonumber\\
&\le c(p)\left[\sup_{f\in F}\int_{-\infty}^{-N}|(\tilde
P_Nf)(x)|^pdx+\sup_{f\in F}\int_{-\infty}^{-N}|\hat
P_Nf)(x)|^pdx\right]\nonumber\\
&=c(p)[\tilde T_1(N)+\hat T_1(N)].\label{4.21} \end{align} Here
\begin{equation}\label{4.22}
\tilde T_1(N)=\sup_{f\in F}\int_{-\infty}^{-N}|(\tilde
P_Nf)(x)|^pdx,\end{equation}
\begin{equation}\label{4.23}
\hat T_1(N)=\sup_{f\in F}\int_{-\infty}^{-N}|(\hat
P_Nf)(x)|^pdx.\end{equation} Clearly, $T_1(N)$ satisfies
\eqref{4.17} if
\begin{equation}\label{4.24}
\tilde T_1(N)\to0,\qquad \hat T_2(N)\to0\qquad\text{as}\quad
N\to\infty.\end{equation}

To prove the first relation of \eqref{4.24}, we use the definition
of the set $F,$ \thmref{thm2.32} and Corollary \ref{cor4.4}:
\begin{align*}
\tilde T_1(N)&=\sup_{f\in F}\int_{-\infty}^{-N}|(\tilde
P_Nf)(x)|^pdx\le\|\tilde P_N\|^p_{L_p(-\infty,-N)\to
L_p(-\infty,-N)}\cdot \sup_{f\in F}\|f\|_{L_p(-\infty,-N)}^p\\
&\le c(p)\sup_{x\le -N}\left[\left(\int_{-\infty}^x
v(t)^pdt\right)^{1/p}\cdot\left(\int_x^{-N}u(\xi)^{p'}d\xi\right)^{1/p'}\right]^p\cdot\sup_{f\in
F}\|f\|_p^p\\
&\le c(p)\sup_{x\le-N}\theta_p^{(-)}(x,N)\to0\qquad\text{as}\quad
N\to \infty.
\end{align*}

To prove the second relation of \eqref{4.24}, we use the
definition of the set $F,$ H\"older's inequality and Corollary
\ref{cor4.3}:
\begin{align*}
\hat T_1(N)&=\sup_{f\in F}\int_{-\infty}^{-N}|(\hat
P_Nf)(x)|^pdx\le \sup_{f\in
F}\left(\int_{-\infty}^{-N}v(x)^pdx\right)\cdot\left(\int_{-N}^\infty
u(\xi)|f(\xi|d\xi\right)^p\\
&\le\left(\int_{-\infty}^{-N}v(x)^pdx\right)\cdot\left(\int_{-N}^\infty
u(\xi)^{p'}d\xi\right)^{p/p'}\cdot\sup_{f\in
F}\|f\|_{L_p(-N,\infty)}^p\\
&\le \theta_p^p(-N)\to 0\qquad\text{as}\quad N\to\infty.
\end{align*}
Thus relation \eqref{4.17} holds, and therefore condition 3) is
satisfied.

\textit{Verification of condition 2)}. According to \eqref{2.32},
it is enough to show that
\begin{equation}\label{4.25}
\lim_{\delta\to0}\sup_{f\in \mathcal
K}\sup_{|t|\le\delta}\|(G_if)(\cdot+t)-(G_if)(\cdot)\|_p=0,\quad
i=1,2.\end{equation}

Both equalities of \eqref{4.25} are checked in the same way;
therefore, below we only consider the case $i=2.$ Furthermore,
equality \eqref{4.25} will be prove as soon as we find
$\delta=\delta(\varepsilon)\in(0,1]$ for a given $\varepsilon>0$
such that
\begin{equation}\label{4.26}
\sup_{f\in\mathcal
K}\sup_{|t|\le\delta}\|(G_2f)(\cdot+t)-(G_2f)(\cdot)\|_[\le
\varepsilon.\end{equation} Thus, let $\varepsilon>0$ be given. Set
$N\ge 1$ (the choice of $N $ will be made more precise later).
Then for $f\in\mathcal K$ we have
\begin{align}
\|(G_2f)(\cdot+t)&-(G_2f)(\cdot)\|_p^p=\|(G_2f)(\cdot+t)-(G_2f)(\cdot)\|_{L_p(-N,N)^+}^p\nonumber\\
&\quad
+\|(G_2f)(\cdot+t)-(G_2f)(\cdot)\|_{L_p(-\infty,-N)}^p+\|(G_2f)(\cdot+t)-(G_2f)(\cdot)\|_{L_p(N,\infty))}^p
\nonumber \\
&\le
\|(G_2f)(\cdot+t)-(G_2f)(\cdot)\|_{L_p(-N,N)}^p+2\|G_2f\|_{L_p(-\infty,-N+1)}^p
\nonumber\\ &\quad +2\|G_2f\|_{L_p(N-1,+\infty)}^p.\label{4.27}
\end{align}

 By 3), for the given $\varepsilon>0$ there exists
 $N_0=N_0(\varepsilon)$ such that
 $$\sup_{f\in\mathcal
 K}\|G_2f\|_{L_p(-\infty,-N_0+1)}^p+\sup_{f\in\mathcal
 K}\|G_2f\|_{L_p(N_0-1,\infty)}\le
 \frac{\varepsilon}{4}^p.$$
 Therefore, for $N=N_0$ inequality \eqref{4.27} can be continued
 as follows:
 \begin{equation}\label{4.28}
 \|(G_2f)(\cdot+t)-(G_2f)(\cdot)\|_p^p\le\|(G_2f)(\cdot+t)-(G_2f)(\cdot)\|_{L_p(-N_0,N_0)}^p+\frac{\varepsilon
 ^p}{2}.\end{equation}

 Throughout the sequel, $|x|\le N_0$ and $|t|\le\delta$ (the
 number $\delta$ will be chosen later). Let us continue estimate
 \eqref{4.28}. We have
 \begin{align}
 |(G_2f)(x+t)&-(G_2f)(x)| =\left|v(x+t)\int_{x+t}^\infty
 u(\xi)f(\xi)d\xi-v(x)\int_x^\infty u(\xi)f(\xi)d\xi\right|\nonumber\\
 &\le |v(x+t)-v(x)|\cdot\left|\int_x^\infty
 u(\xi)f(\xi)d\xi\right|+v(x+t)\left|\int_x^{x+t}u(\xi)f(\xi)d\xi\right|\nonumber\\
 &:=(\mathcal A f)(x,t)+(Bf)(x,t).\label{4.29}\end{align}
 Here
 \begin{equation}\label{4.30}
 (Af)(x,t)=|v(x+t)-v(x)|\cdot\left|\int_x^\infty
 u(\xi)f(\xi)d\xi\right|,\quad f\in\mathcal K,\end{equation}
 \begin{equation}\label{4.31}
 (Bf)(x,t)= v(x+t)) \cdot\left|\int_x^{x+t}
 u(\xi)f(\xi)d\xi\right|,\quad f\in\mathcal K.\end{equation}

 Let us introduce the numbers
 \begin{equation}\label{4.32}
 \delta_1=\min_{x\in[-N_0,N_0]}d(x),\qquad \eta=\sup_{x\in[-N_0,N_0]}\sup_{|t|\le\delta} \left|\int_x
 ^{x+t}\frac{d\xi}{r(\xi)h(\xi)}\right|.\end{equation}
 From absolute continuity of the Lebesgue integral, it follows
 that given $\varepsilon>0,$ one can choose
 $\delta=\delta(\varepsilon) $ so small that the following
 inequalities  hold:
 \begin{equation}\label{4.33}
 \delta\le\delta_1,\qquad \eta\le\frac{\varepsilon}{\alpha}.\end{equation}
 (Here $\alpha$ is a positive number to be chosen later.)

 In the following estimate of $(Af)(x,t)$, we use
 \eqref{4.33}, the equalities (see \cite{10})
 \begin{equation}\label{4.34}
 \frac{v'(x)}{v(x)}=\frac{1+r(x)\rho'(x)}{2r(x)\rho(x)},\qquad \frac{u'(x)}{u(x)}= -
 \frac{1-r(x)\rho'(x)}{2r(x)\rho(x)},\qquad x\in\mathbb R,\end{equation}
 and estimates \eqref{2.17}, \eqref{2.25} and \eqref{2.26}:
 \begin{align}
 (Af)(x,t)&=|v(x+t)-v(x)|\cdot\left|\int_x^\infty
 u(\xi)f(\xi)d\xi\right|=\left|\int_x^{x+t}v'(s)ds\right|\cdot\frac{1}{v(x)}|(G_2f)(x)|\nonumber\\
 &=\left|\int_x^{x+t}\frac{r(s)v'(s)}{v(s)}
\cdot
\frac{v(s)}{v(x)}\cdot\frac{ds}{r(s)}\right|\cdot|(G_2f)(x)|\nonumber\\
 &\le\left|\int_x^{x+t}\frac{2}{\rho(s)}\cdot
 e^2\frac{ds}{r(s)}\right|\cdot|(G_2f)(x)|\le
 c\left|\int_x^{x+t}\frac{ds}{r(s)h(s)}\right|\cdot|(G_2f)(x)|\nonumber\\ &\le
 \frac{c\varepsilon}{\alpha}\cdot|(G_2f)(x)|.\label{4.35}
 \end{align}

 Furthermore, in the estimate of $(Bf)(x,t)$ we use \eqref{2.25},
 \eqref{2.26}, \eqref{4.33}, H\"older's inequality and the
 definition of the set $\mathcal K:$
 \begin{align}
 (Bf)(x,t)&=v(x+t)  \left|\int_x^{x+t}
 u(\xi)f(\xi)d\xi\right|\nonumber\\
 &\le \frac{v(x+t)u(x+t)}{v(x)u(x)}\cdot \rho(x)\left|\int_x^{x+t}\frac{u(\xi)}{u(x)}\cdot
 \frac{u(x)}{u(x+t)}\cdot |f(\xi)|d\xi\right|\nonumber\\
 &\le c\rho(x)\left|\int_x^{x+t}|f(\xi)|d\xi\right|\le c\rho(x)|t|^{1/p'}\cdot\|f\|_p\nonumber\\
 &\le c\rho(x)\delta^{1/p'}\le c\big(\max_{|x|\le N_0}\rho(x)\big)\cdot\delta^{1/p'}.\label{4.36}
 \end{align}

 The following estimates are derived from \eqref{4.35},
 \eqref{4.36}, the definition of the set $\mathcal K$ and
 \eqref{2.35}:
  \begin{align*}
 |(G_2f)(x+t)-(G_2f(x)|&\le (\mathcal A f)(x,t)+(Bf)(x,t) r\\
 &\le \frac{c\varepsilon}{\alpha}|(G_2f)(x)|+c\big(\max_{|x|\le N_0}\rho(x)\big)\delta^{1/p'}\ \Rightarrow
 \\
  \|(G_2f)(\cdot+t)-(G_2f)(\cdot)\|_{L_p(-N_0,N_0)}
 &\le \frac{c\varepsilon}{\alpha} \|(G_2f\|_p+c\big(\max_{|x|\le N_0}\rho(x)\big)N_0^{1/p}
 \cdot \delta^{1/p'}\\
 &\le \frac{cB}{\alpha}\varepsilon+c\big(\max_{|x|\le N_0}\rho(x)\big)N_0^{1/p}\delta^{1/p'}.
 \end{align*}

 Set $\alpha=2^{1+\frac{1}{p}}\cdot cB$ and, if necessary, choose
 a smaller $\delta$ so that the following inequality holds:
 $$c\big(\max_{|x|\le N_0}\rho(x)\big)\cdot
 N_0^{1/p}\delta^{1/p'}\le\frac{\varepsilon}{2^{1+1/p}}.$$
 Then we get the estimates
 $$\|(G_2f)(\cdot+t)-(G_2f)(\cdot)\|_{L_p(-N_0,N_0 )}\le
 \frac{\varepsilon}{2^{1/p}}\ \Rightarrow\ \text{(see
 \eqref{4.28})},$$
  $$\|(G_2f)(\cdot+t)-(G_2f)(\cdot)\|_p^p\le
 \frac{\varepsilon^p}{2 }+\frac{\varepsilon^p}{2}=\varepsilon^p\
 \Rightarrow\
 \eqref{4.25}\ \Rightarrow 2).$$

 The theorem is proved.
\end{proof}

\renewcommand{\qedsymbol}{}
\begin{proof}[Proof of \thmref{thm3.4}]
We need the following assertion.
\end{proof}

\begin{lem}\label{lem4.5}
Suppose that condition \eqref{3.1} holds. Then $B<\infty$ (see
\eqref{2.34}).\end{lem}

 \renewcommand{\qedsymbol}{\openbox}
\begin{proof}
{}From \eqref{3.1} and \eqref{2.14} it follows that
$\rho(x)d(x)\to0$ as $|x|\to\infty.$ Hence there is $x_0\gg1$ such
that $\rho(x)d(x)\le 1$ for $|x|\ge x_0.$ By \lemref{lem2.10}, the
function $\rho(x)d(x)$ is continuous for $x\in\mathbb R$ and is
therefore bounded on $[-x_0,x_0]$. Hence $S<\infty$ (see
\eqref{2.36}), and therefore $B<\infty$ (see
\eqref{2.34}).\end{proof}

Let us now go to the assertion of the theorem. Since
$G(x,t)=G(t,x)$ for all $t,x\in\mathbb R$ (see \eqref{2.9}), the
operator $G:L_2\to L_2$ is symmetric and bounded (see
\lemref{lem4.5} and \eqref{2.35}). Hence the operator $G$ is
self-adjoint and, by \thmref{thm3.1}, compact. Furthermore,
estimates \eqref{3.3} follow from positivity of $G$ which, in
turn, will be proved below. Towards this end, we need the
following two lemmas.

\begin{lem}\label{lem4.6}
The equalities
\begin{equation}\label{4.37}
\lim_{|x|\to\infty}\frac{u(x)}{v(x)}\cdot\int_{-\infty}^xv(t)^2dt=0,
\end{equation}
\begin{equation}\label{4.38}
\lim_{|x|\to\infty}\frac{v(x)}{u(x)}\cdot\int^{\infty}_xu(t)^2dt=0
\end{equation}
hold if and only if condition \eqref{3.1} is satisfied.
\end{lem}

\renewcommand{\qedsymbol}{}
\begin{proof}[Proof of \lemref{lem4.6}]  Necessity.
Both equalities are checked in the same way, and therefore below
we only consider \eqref{4.38}. Below $x\in\mathbb R,$ and we apply
estimates \eqref{2.25} and \eqref{2.14}:
\begin{align*}
I(x)&\doe\frac{v(x)}{u(x)}\int_x^\infty u^2(t)dt\ge
\frac{v(x)}{u(x)}\cdot\int_x^{x+d(x)}u^2(t)dt\\
&=\frac{v(x)}{u(x)}\int_x^{x+d(x)}\left(\frac{u(t)}{u(x)}\right)^2\cdot
u^2(x)dt\ge c^{-1}\rho(x)d(x)\ge c^{-1}h(x)d(x)>0.\end{align*} It
remains to refer to \eqref{4.38}.

\renewcommand{\qedsymbol}{\openbox}

\begin{proof}[Proof of \lemref{lem4.6}]  Sufficiency.
{}From \eqref{2.7} we obtain the equality
\begin{equation}\label{4.39}
I(x)=\frac{v(x)}{u(x)}\cdot\int_x^\infty u^2(t)dt=\int_x^\infty
\rho(t)\exp\left(-\int_x^t\frac{d\xi}{r(\xi)\rho(\xi)}\right)dt.\end{equation}
Let $x\to\infty.$ Below we use \eqref{4.39}, properties of an
$\mathbb R(x,d)$-covering of $[x,\infty),$ \eqref{2.26} and
\eqref{4.3}:
\begin{align*}
I(x)&=\sum_{n=1}^\infty\int_{\Delta_n}\rho(t)\exp\left(-\int_x^t\frac{d\xi}{r(\xi)\rho(\xi)}\right)dt\
 \le
c\sum_{n=1}^\infty\rho(x_n)d(x_n)\exp\left(-\int_{\Delta_1^-}^{\Delta_n^-}\frac{d\xi}{r(\xi)\rho(\xi)}\right)\\
&\le c\sum_{n=1}^\infty
h(x_n)d(x_n)\exp\left(-\frac{1}{2}\int_{\Delta_1^-}^{\Delta_n^-}\frac{d\xi}{r(\xi)h(\xi)}\right)\\
&\le c\sup_{t\ge
x}(h(t)d(t))\sum_{n=1}^\infty\exp\left(-\frac{n-1}{2}\right)=c\sup_{t\ge
x}(h(t)d(t)).\end{align*} The latter inequality and \eqref{3.1}
imply \eqref{4.38} (as $x\to\infty).$ Let now $x\to-\infty.$ Fix
$\varepsilon>0$ and choose $\ell=\ell(\varepsilon)\gg1$ so that
the following estimate will hold
\begin{equation}\label{4.40}
4c_0B \ell^2\cdot\exp\left(-\frac{\ell-1}{2}\right)\le
\varepsilon,\qquad
c_0=\sum_{k=1}^\infty\exp\left(-\frac{k-1}{2}\right).\end{equation}

Consider the segments $\{\Delta_k\}_{k=1}^\ell$ from an $\mathbb
R(x,d)$-covering of $[x,\infty).$ Let us show that
\begin{equation}\label{4.41}
\lim_{x\to-\infty}\Delta_\ell^+=-\infty.\end{equation} Assume the
contrary: there exists $c>-\infty$ such that $\Delta_\ell^+\ge c$
as $x\to-\infty.$ Then by \eqref{2.10} and \eqref{4.3}, we have
$$\ell=\int_{\Delta_1^-}^{\Delta_\ell^+}\frac{d\xi}{r(\xi)h(\xi)}\ge\int_x^c\frac{d\xi}{r(\xi)h(\xi)}
\ge\frac{1}{2}\int_x^c\frac{d\xi}{r(\xi)\rho(\xi)}\to\infty\quad\text{as}\quad
x\to-\infty,$$ a contradiction, so \eqref{4.41} is proved.

Let us now choose $x_1(\varepsilon)$ and $x_2(\varepsilon)$ so
that the following inequalities will hold:
\begin{alignat}{2}
4e^2c_0\cdot\ell\cdot h(t)d(t) &\le\varepsilon \quad
&&\text{for}\ t\le -x_1(\varepsilon),\label{4.42}\\
\Delta_\ell^+&\le-x_1(\varepsilon)\quad &&\text{for}\ x\le
-x_2(\varepsilon).\label{4.43} \end{alignat} Let
$x_0=\max\{x_1(\varepsilon),x_2(\varepsilon)\}.$ Below for $x\le
x_0$ we use \eqref{4.39}, properties of an $\mathbb
R(x,d)$-covering of $\mathbb R,$ \eqref{2.26}, \eqref{4.42},
\eqref{4.43} and \eqref{4.40}:
\begin{align*}
I(x)&=\sum_{n=1}^\infty\int_{\Delta_n}\rho(t)\exp\left(-\int_x^t\frac{d\xi}{r(\xi)\rho(\xi)}\right)\\
&\le 2e^2\left\{\sum_{n=1}^\ell
h(x_n)d(x_n)\exp\left(-\frac{n-1}{2}\right)+\sum_{n=\ell+1}^\infty
h(x_n)d(x_n)\exp\left(-\frac{n-1}{2}\right)\right\}\\
 &\le
 2e^2c_0\sup_{t\le\Delta_\ell^+}(h(t)d(t))+2e^2c_0B\exp(-\frac{\ell-1}{2} )\le
 \frac{\varepsilon}{2}+\frac{\varepsilon}{2}=\varepsilon.
\end{align*}
The obtained estimates lead to \eqref{4.38}.
\end{proof}

\begin{lem}\label{lem4.7}
Let $p\in(1,\infty),$ $f\in L_p$ and $y=Gf.$ Then, if condition
\eqref{3.1} holds, we have
\begin{equation}\label{4.44}
\lim_{|x|\to\infty }r(x)y'(x)y(x)=0.\end{equation}
\end{lem}

\begin{proof} From \eqref{3.1} and \lemref{lem4.5} it follows that
$B<\infty.$ Let, for example, $x\to\infty$ (the case $x\to-\infty$
is treated in a similar way). Below we use the definition and
properties of the operator $G: L_2\to L_2$ (see \eqref{2.29}),
\eqref{2.17}, \eqref{2.30}--\eqref{2.33} and the Schwarz
inequality:
\begin{align*}
r(x)|y'(x)|\cdot|y(x)|&\le
(G|f|)(x)\cdot\left[r(x)\left|\frac{d}{dx}(Gf)(x)\right|\right]\\
&\le
(G|f|)(x)\cdot\left[\frac{r(x)|u'(x)|}{u(x)}\cdot(G_1|f|)(x)+\frac{r(x)v'(x)}{v(x)}(G_2|f|)(x)\right]\\
&\le \frac{[(G|f|)(x)]^2}{\rho(x)}\le c\frac{[(G_1|f|)(x)]^2
+[(G_2|f|)(x)]^2}{\rho(x)}
\\
&\le c\left\{\frac{u(x)}{v(x)}\cdot\int_{-\infty}^x
v^2(t)dt+\frac{v(x)}{u(x)}\cdot\int_x^\infty
u^2(t)dt\right\}\cdot\|f\|_2^2.\end{align*} It remains to apply
\lemref{lem4.6}
\end{proof}

Let us now complete the proof of the theorem. Below we assume that
$f\in L_2$ and $y:=Gf.$ Then, obviously, $f=\mathcal L_2y$, and we
have the relations
\begin{align*}
\int_{-\infty}^\infty(Gf)(x)\cdot\bar
f(x)dx&=\int_{-\infty}^\infty y(x)\overline{(\mathcal
L_2y)(x)}dx=\lim_{\substack{ b\to\infty\\
a\to-\infty}}
\int_a^by(x)\overline{[-(r(x)y'(x))'+q(x)y(x)]}dx\\
&=\lim_{\substack{ b\to\infty\\
a\to-\infty}}\int_a^by(x)\big[-(r(x)\overline{y'(x)}\big]'+q(x)\overline{y(x)}dx\\
&=\lim_{\substack{ b\to\infty\\
a\to-\infty}}\left[-r(x)\overline{y'(x)}y(x)\Big|_a^b+\int_a^b(r(x)|y'(x)|^2+q(x)|y(x)|^2)dx\right]\\
&=\int_{-\infty}^\infty(r(x)|y'(x)|^2+q(x)|y(x)|^2)dx\ge0.
\end{align*}
\end{proof}

\begin{proof}[Proof of Corollary \ref{cor3.7}] The following
relations are based on \thmref{thm2.1}:
$$
\left.\begin{array}{ll}
%\qquad
\displaystyle{r(x)v'(x)-r(t)v'(t)=\int\limits_t^xq(\xi)v(\xi)d\xi,\quad t\le x\in\mathbb R }\\
\displaystyle{r(t)u'(t)-r(x)u'(x)=\int\limits_x^tq(\xi)u(\xi)d\xi\quad\
t\ge x\in\mathbb R}\end{array}\right\}\quad\Rightarrow
$$
$$r(x)v'(x)\ge \int_{-\infty}^x q(\xi)v(\xi)d\xi,\quad
-r(x)u'(x)\ge \int_x^\infty q(\xi)u(\xi)d\xi\quad\Rightarrow$$
\begin{align*} 1 =r(x)[v'(x)u(x)-u'(x)v(x)]&\ge u(x)\int_{-\infty}^x
q(t)v(t)dt+v(x)\int_x^\infty q(t)u(t)dt\\
&=\int_{-\infty}^\infty q(t)G(x,t)dt.
\end{align*}

Below we continue the last inequality using \eqref{2.9},
\eqref{2.14}, \eqref{2.19} and \eqref{2.26}: \begin{align}
1&\ge\int_{x-d(x)}^{x+d(x)}q(t)G(x,t)dt=\int_{x-d(x)}^{x+d(x)}\sqrt{\rho(t)\rho(x)}\exp\left(-\frac{1}
{2}\left|\int_x^t\frac{d\xi}{r(\xi)\rho(\xi)}\right|\right)dt\nonumber\\
&\ge
c^{-1}h(x)\exp\left(-\frac{1}{4}\int_{x-d(x)}^{x+d(x)}\frac{d\xi}{r(\xi)h(\xi)}\right)\cdot
\int_{x-d(x)}^{x+d(x)}q(t)dt
=c^{-1}h(x)\int_{x-d(x)}^{x+d(x)}q(t)dt.\label{4.45}
\end{align}

Equation \eqref{1.1} is correctly solvable in $L_p,$
$p\in(1,\infty)$ since $\mathcal A>0$ (see \thmref{thm2.25}), and
from \eqref{4.45} it follows that
$$c(\mathcal A(x))^{-1}\ge h(x)d(x),\qquad x\in\mathbb R\
\Rightarrow\ \eqref{3.1}.$$ The assertion now follows from
\thmref{thm3.1}.
\end{proof}

\begin{proof}[Proof of Corollary \ref{cor3.8}]
Since $q(x)\to\infty$ as $|x|\to\infty,$ condition \eqref{1.3}
holds, and therefore all auxiliary functions are defined (see
Lemmas \ref{lem2.5} and \ref{lem2.10}). Let $q(x)\ge 1$ for
$|x|\ge x_1.$ Then by \eqref{2.22}, there exists $x_2\gg x_1$ such
that for $|x|\ge x_2$ we have
\begin{equation}\label{4.46}
[x-d(x),x+d(x)]\cap[[-x_1,x_1]=\emptyset.\end{equation} Then for
$|x|\ge x_2$ from \eqref{4.45} it follows that
$$c\ge h(x)\int_{x-d(x)}^{x+d(x)} q(t)dt\ge
h(x)\int_{x-d(x)}^{x+d(x)}1 dt=2h(x)d(x)\ \Rightarrow$$
$$\sup_{|x|\ge x_2}(h(x)d(x))\le 2c<\infty.$$
Since the function $h(x)d(x)$ is bounded on $[-x_2,x_2]$ (see the
proof of \lemref{lem4.5}), we have $B<\infty,$ and by
\thmref{thm2.21} equation \eqref{1.1} is correctly solvable in
$L_p,$ $p\in(1,\infty).$ Further, from \eqref{2.22} it follows
that
$$c(\mathcal A(x))^{-1}\ge h(x)d(x),\quad |x|\ge x_2;\quad
\mathcal A(x)\to\infty\quad \text{as}\qquad |x|\to\infty.$$ Hence
condition  \eqref{3.1} holds, and the assertion of the corollary
follows from \thmref{thm3.1}.
\end{proof}

\begin{proof}[Proof of Corollary \ref{cor3.9}]
We need the following fact whose proof is presented for the sake
of completeness.

\begin{lem}\label{lem4.8} \cite{17} Suppose that conditions
\eqref{1.2}--\eqref{1.3} hold and $r\equiv1.$ Then equality
\eqref{3.5} holds if and only if $\tilde d(x)\to0$ as
$|x|\to\infty$ (see \eqref{2.15}).\end{lem}

 \renewcommand{\qedsymbol}{}
\begin{proof}[Proof of \lemref{lem4.8}]  Necessity. Assume the
contrary: equality \eqref{3.5} holds but $\tilde d(x)\nrightarrow
0$ as $|x|\to\infty.$ This means that there exist $\varepsilon>0$
and points $\{x_n\}_{n=1}^\infty$ such that $|x_n|\to\infty$ as
$n\to\infty$ and $\tilde d(x_n)\ge\varepsilon.$ This implies
$$\frac{1}{\varepsilon}\ge\frac{1}{\tilde
d(x_n)}=\frac{1}{2}\int_{x_n-\tilde d(x_n)}^{x_n+\tilde
d(x_n)}q(t)dt\ge
\frac{1}{2}\int_{x_n-\varepsilon}^{x_n+\varepsilon}q(t)dt,\quad
n\ge1.$$ Thus equality \eqref{3.5} breaks down for
$a=\varepsilon,$ a contradiction.
\end{proof}

\renewcommand{\qedsymbol}{\openbox}

\begin{proof}[Proof of \lemref{lem4.8}]  Sufficiency.
If $\tilde d(x)\to0$ as $|x|\to\infty,$ then for any
$a\in(0,\infty)$ and for all $|x|\gg1,$ we have
$$\frac{1}{2}\int_{x-a}^{x+a}q(t)dt\ge\frac{1}{2}\int_{x-\tilde
d(x)}^{x+\tilde d(x)} q(t)dt=\frac{1}{\tilde d(x)}\ \Rightarrow\
\eqref{3.5}.$$
\end{proof}

Let us now go to the corollary. For $r\equiv1,$ from \eqref{2.10}
and \eqref{2.26} we obtain
$$
\left.\begin{array}{ll}
\displaystyle{1=\int_{x-d(x)}^{x+d(x)}\frac{dt}{h(t)}\le c\frac{d(x)}{h(x)} }\\
\displaystyle{1=\int_{x-d(x)}^{x+d(x)}\frac{dt}{h(t)}\le
c^{-1}\frac{d(x)}{h(x)}}\end{array}\right\}\quad\Rightarrow\quad
h(x)\asymp d(x),\qquad x\in\mathbb R.
$$
On the other hand, from \eqref{2.16} and \eqref{2.14}, it follows
that $h(x)\asymp \rho(x)\asymp\tilde d(x),$ $x\in\mathbb R.$
Putting this together, we obtain the main relations: $h(x)\asymp
d(x)\asymp\tilde d(x),$ $x\in\mathbb R.$ Further, as $m(a_0)>0$
for some $a_0\in(a,\infty)$, we conclude that equation \eqref{1.1}
is correctly solvable in $L_p,$ $p\in(1,\infty)$ by
\thmref{thm2.24}. We have $h(x)d(x)\to0$ as $|x|\to\infty$ if and
only if $\tilde d(x)\to0$ as $|x|\to\infty$ since
$h(x)d(x)\asymp\tilde d^2(x),$ $x\in\mathbb R$. The assertion of
the corollary now follows from \lemref{lem4.8}.
\end{proof}

\renewcommand{\qedsymbol}{\openbox}

\begin{proof}[Proof of Corollary \ref{cor3.10}] By
\thmref{thm2.30}, in all the following cases 1)--3), equation
\eqref{1.1} is correctly solvable in $L_p,$ $p\in(1,\infty).$ Let
us show that in the same cases condition \eqref{3.1} holds, and
thus by \thmref{thm3.1} our assertion will then be proved.

1) Let $x\in\mathbb R,$ $\Delta(x)=[x-d(x),x+d(x)]$. Below we use
the Schwarz inequality and \eqref{2.19}:
\begin{align*}
2d(x)&=\int_{\Delta(x)}\sqrt{\frac{r(t)h(t)}{r(t)h(t)}}dt\le
\left(\int_{\Delta(x)}r(t)h(t)dt\right)^{1/2}\cdot\left(\int_{\Delta(x)}\frac{dt}{r(t)h(t)}\right)^{1/2}
\\ &=\left(\int_{\Delta(x)}r(t)h(t)dt\right)^{1/2}\quad \Rightarrow\end{align*}
\begin{equation}\label{4.47} 4d^2(x)\le \int_{\Delta(x)} r(t)h(t)dt,\quad
x\in\mathbb R.\end{equation} Let
$\eta(x)=\sup_{t\in\Delta(x)}(r(t)h^2(t)).$ From \eqref{2.22} and
\eqref{3.6} it follows that $\eta(x)\to0$ as $|x|\to\infty.$
Further, from \eqref{4.47}  using \eqref{2.26}, we obtain
 $$4d^2(x)\le
 \int_{\Delta(x)}r(t)h^2(t)\cdot\frac{h(x)}{h(t)}\frac{dt}{h(x)}\le
c\eta(x)\frac{d(x)}{h(x)}\quad \Rightarrow$$
$$0< h(x)d(x)\le c\eta(x),\qquad x\in\mathbb
R\qquad\Rightarrow\qquad\eqref{3.1}.$$

2) This assertion follows from 1) and \eqref{2.12}, \eqref{3.7}
and \eqref{3.6}:
$$r(x)h^2(x)=r(x)\varphi(x)\psi(x)\frac{\varphi(x)\psi(x)}{(\varphi(x)+\psi(x))^2}\le
r(x)\varphi(x)\psi(x),\quad x\in\mathbb R.$$

3) From \eqref{2.22} it follows that $d(x)\le |x|$ for all
$|x|\gg1$. Hence $0<h(x)d(x)\le h(x)|x|$ for all $|x|\ge 1,$ and
therefore \eqref{3.1} holds because of \eqref{3.8}.
\end{proof}

\begin{proof}[Proof of Corollary \ref{cor3.11}] For $x\in\mathbb
R,$ according to \eqref{2.19} and \eqref{2.26}, we have
$$1=\int_{\Delta(x)}\frac{dt}{r(t)h(t)}\ge
\frac{c^{-1}}{h(x)}\int_{\Delta(x)}\frac{dt}{r(t)}\ge
\frac{c^{-1}}{r_0}\frac{d(x)}{h(x)}\quad\Rightarrow$$
$$0<h(x)d(x)\le ch^2(x).$$

Then $B\le ch_0^2<\infty,$ and condition \eqref{3.1} holds. The
assertion follows from Theorems \ref{thm2.21} and \ref{thm3.1}.
\end{proof}

\begin{proof}[Proof of \thmref{thm3.13}]
Since $\theta(x)\to 0$ as $|x|\to\infty,$ we have $r^{-1}\in
L_1(\mathbb R)$ and $\theta<\infty$ (see \eqref{2.46}). We now
need the following lemma.

\begin{lem}\label{lem4.9} Suppose that conditions \eqref{1.2} and
\eqref{2.1} hold and $r^{-1}\in L_1.$ Then we have the equality
\begin{equation}\label{4.48}
\rho(x)\le
\tau\int_{-\infty}^x\frac{dt}{r(t)}\cdot\int_x^\infty\frac{dt}{r(t)},\qquad
x\in\mathbb R.
\end{equation}
Here
\begin{equation}\label{4.49}
\tau=\max\left\{\left(\int_{-\infty}^0\frac{dt}{r(t)}\right)^{-1},\left(\int_0^\infty\frac{dt}{r(t)}\right
)^{-1}\right\}.\end{equation}
\end{lem}

\begin{proof}
From \thmref{thm2.1}, it easily follows that
\begin{equation}\label{4.50}
u(x)=v(x)\int_x^\infty\frac{dt}{r(t)v^2(t)},\quad
v(x)=u(x)\int_{-\infty}^x\frac{dt}{r(t)u^2(t)},\qquad x\in\mathbb
R.
\end{equation}
{}From \eqref{4.50} and \eqref{2.3} we now obtain
$$\rho(x)=v^2(x)\int_x^\infty\frac{dt}{r(t)v^2(t)}\le\int_x^\infty\frac{dt}{r(t)},\qquad
x\in\mathbb R,$$
$$\rho(x)u^2(x)\int_{-\infty}^x\frac{dt}{r(t)u^2(t)}\le\int_{-\infty}^x\frac{dt}{r(t)},\qquad
x\in \mathbb R.$$ Hence
\begin{equation}\label{4.51}
\rho(x)=\begin{cases}
\displaystyle{\int_x^\infty\frac{dt}{r(t)}},\quad & \text{if}\
x\ge 0\\ \\ \displaystyle{ \int_{-\infty}^x\frac{dt}{r(t)}},\quad
& \text{if}\ x\le 0\end{cases}\end{equation} Estimate \eqref{4.48}
follows from\eqref{4.51} and \eqref{4.49}.
\end{proof}
Further, from \eqref{2.23} we conclude that $s(x)\le|x|$ for all
$|x|\gg1$, and therefore
$$\rho(x)s(x)\le\tau|x|\cdot\int_{-\infty}^x
\frac{dt}{r(t)}\cdot\int_x^\infty
\frac{dt}{r(t)}=\tau\theta(x),\qquad |x|\gg1.$$ The latter
inequality means that $S\le\tau\theta<\infty.$ Hence equation
\eqref{1.1} is correctly solvable in $L_p,$ $p\in(1,\infty)$ by
\thmref{thm2.22}. If $\theta(x)\to0$ as $|x|\to\infty,$ then
condition \eqref{3.2} holds, and the operator $G: L_p\to L_p$,
$p\in(1,\infty)$, is compact by \thmref{thm3.2}
\end{proof}

\begin{proof}[Proof of \thmref{thm3.14}]
Below we follow the scheme of the proof of Corollary
\ref{cor3.10},1). Let $x\in\mathbb R,$ $\tilde
\Delta(x)=[x-s(x+s(x)]$ (see \eqref{2.19}). From the Schwarz
inequality and \eqref{2.19}, we get
\begin{align*}
2s(x)&=\int_{\tilde
\Delta(x)}\sqrt{\frac{r(t)\rho(t)}{r(t)\rho(t)}}dt\le\bigg(\int_{\tilde\Delta(x)}r(t)\rho(t)dt\bigg)^{1/2}
\bigg(\int_{\tilde\Delta(x)}\frac{dt}{r(t)\rho(t)}\bigg)^{1/2}\\
&=\bigg(\int_{\tilde\Delta(x)}r(t)\rho(t)dt\bigg)^{1/2},\quad
x\in\mathbb R\quad\Rightarrow
\end{align*}
\begin{equation}\label{4.52}
4s^2(x)\le\int_{\tilde\Delta(x)}r(t)\rho(t)dt,\qquad x\in\mathbb
R.
\end{equation}
Further, since $\nu<\infty,$ we have $r^{-1}\in L_1.$ Therefore by
\lemref{lem4.9} we have estimate \eqref{4.48}. This implies the
inequality
\begin{equation}\label{4.53}
r(x)\rho^2(x)\le c\nu(x),\qquad x\in\mathbb R.\end{equation}

Since $\nu<\infty,$ from \eqref{4.52}, \eqref{4.53} and
\eqref{2.28}, we get \begin{gather} 4s^2(x)\le
\int_{\tilde\Delta(x)}r(t)\rho^2(t)\frac{\rho(x)}{\rho(t)}\le c\nu
\frac{s(x)}{\rho(x)},\quad x\in\mathbb R\quad
\Rightarrow\nonumber\\
s(x)\rho(x)\le c\nu,\quad x\in\mathbb R\quad\Rightarrow\quad S\le
c\nu.\label{4.54}\end{gather} From \eqref{4.54} and
\thmref{thm2.22} it follows that equation \eqref{1.1} is correctly
solvable in $L_p,$ $p\in(1,\infty)$. Let $\nu(x)\to0$ as
$|x|\to\infty.$ Then by \eqref{4.53} and \eqref{2.23}, we also
have $\tilde\eta(x)\to0$ as $|x|\to\infty,$ where
$\tilde\eta(x)=\sup\limits_{t\in\tilde\Delta(x)}r(t)\rho^2(t).$
Hence
  $$
  4s^2(x)\le\int_{\tilde\Delta(x)}r(t)\rho^2(t)\frac{\rho(x)}{\rho(t)}\frac{dt}{\rho(x)}\le
  c\tilde\eta(x)\frac{s(x)}{\rho(x)}\quad\Rightarrow
 $$
$$\rho(x)s(x)\le c\tilde\eta(x),\quad x\in\mathbb
 R\quad\Rightarrow\quad \lim_{|x|\to\infty}\rho(x)s(x)=0.
$$ Thus
the operator $G: L_p\to L_p,$ $p\in(1,\infty),$ is compact by
\thmref{thm3.2}.
\end{proof}

\section{Additional assertions. Example}

Below we consider equation \eqref{1.1} with coefficients
\begin{equation}\label{5.1}
r(x)=e^{\alpha|x|},\qquad q(x)=e^{\beta|x|},\qquad x\in\mathbb
R\end{equation} where $\alpha$ and $\beta$ are any given real
numbers. In what follows, for brevity we refer to it as equation
\eqref{5.1}.

Our goal in connection to \eqref{5.1} is to obtain for this
equation a complete solution of problems 1)--II) and I)--III). As
mentioned above, to study concrete equations \eqref{1.1}, one
needs assertions that allow us to obtain  sharp by order two-sided
estimates of the functions $h$ and $d$ (see Remark \ref{rem2.29}).
Below we will see that getting such inequalities is a certain
technical problem of local analysis. However, the statement of
such a problem depends on the properties of the coefficients of
equation \eqref{1.1}. Therefore, here we restrict ourselves to
considering statements ``sufficient" for investigation of
\eqref{5.1}. (Cf. \cite{10} where estimates of $h$ and $d$ were
obtained for equations \eqref{1.1} with nonsmooth and oscillating
coefficients $r$ and~$q.$)

The next theorem contains a general method that guarantees
obtaining estimates for $h$ and $d.$ Note that this statement is a
formalization of certain devices which were first used by Otelbaev
for estimating his auxiliary functions (see \cite{9}).

\begin{thm}\label{thm5.1} \cite{2} Suppose that conditions
\eqref{1.2} and \eqref{1.3} hold. For a given $x\in\mathbb R$
introduce functions in $\eta\ge0:$
\begin{equation}\label{5.2}
F_1(\eta)=\int_{x-\eta}^x\frac{dt}{r(t)}\cdot\int_{x-\eta}^x
q(t)dt,
\end{equation}
\begin{equation}\label{5.3}
F_2(\eta)=\int_x ^{x+\eta}\frac{dt}{r(t)}\cdot\int_{x} ^{x+\eta}
q(t)dt,
\end{equation}
\begin{equation}\label{5.4}
F_4(\eta)=\int_{x-\eta} ^{x+\eta}\frac{dt}{r(t)h(t)}.
\end{equation}
Then the following assertions hold (see Lemmas \ref{lem2.5} and
\ref{lem2.10}):
 \begin{enumerate}
 \item[1)] the inequality $\eta\ge d_1(x)$ $(0\le\eta\le d_1(x))$
 holds if and only if $F_1(\eta)\ge 1$ $(F_1(\eta)\le 1);$ \item[2)]
the inequality $\eta\ge d_2(x)$ $(0\le\eta\le d_2(x))$ holds if
and only if $F_2(\eta)\ge 1$ $(F_2(\eta)\le 1)$; \item[3)] the
inequality $\eta\ge d(x)$ $(0\le\eta\le d(x))$ holds if and only
 if $F_3(\eta)\ge 1$ $(F_3(\eta)\le 1).$
\end{enumerate}
\end{thm}

The next theorem is an example of using \thmref{thm5.1}.

\begin{thm}\label{thm5.2}
Suppose that the following conditions hold:
\begin{equation}\label{5.5}
r>0,\quad q>0,\quad r\in\mathcal A C^{\loc}(\mathbb R),\quad
q\in\mathcal A C^{\loc}(\mathbb R) \end{equation} (here $\mathcal
A C^{\loc}(R)$ is the set of functions absolutely continuous on
every finite interval of the real axis). Let, in addition,
$$\varkappa_1(x)\to0,\quad \varkappa_2(x)\to0\qquad\text{as}\quad
|x|\to\infty$$ where
\begin{equation}\label{5.6}
\varkappa_1(x)=r(x)\sup_{|t|\le 80\hat
d(x)}\left|\int_x^{x+t}\frac{r'(\xi)}{r^2(\xi)}d\xi\right|,\quad
x\in\mathbb R,\end{equation}
\begin{equation}\label{5.8}
\varkappa_2(x)=\frac{1}{q(x)}\cdot\sup_{|t|\le 80\hat
d(x)}\left|\int_x^{x+t}q'(\xi) d\xi\right|,\quad x\in\mathbb
R,\end{equation}
\begin{equation}\label{5.9}
\hat d(x)=\sqrt{\frac{r(x)}{q(x)}},\qquad x\in\mathbb R.
\end{equation}
 Then for all $|x|\gg1$ each of the equations \eqref{2.11} has a
 unique finite positive solution $d_1(x)$ and $d_2(x)$,
 respectively, and we have (see \eqref{2.12}, \eqref{2.19}):
 \begin{equation}\label{5.10}
 \lim_{|x|\to\infty}\frac{d_1(x)}{\hat
 d(x)}=\lim_{|x|\to\infty}\frac{d_2(x)}{\hat
 d(x)}=1,\end{equation}
 \begin{equation}\label{5.11}
 \lim_{|x|\to\infty}\varphi(x)\sqrt{r(x)q(x)}=\lim_{|x|\to\infty}\psi(x)
\sqrt{r(x)q(x)}=1,\end{equation}
 \begin{equation}\label{5.12}
 \lim_{|x|\to\infty}h(x)\sqrt{r(x)q(x)}=\frac{1}{2},
 \end{equation}
 \begin{equation}\label{5.13}
 c^{-1}\hat d(x)\le d(x)\le c\hat d(x),\qquad x\in\mathbb R.\end{equation}

In addition, $B<\infty$ (see \eqref{2.34} if and only if
$\inf_{x\in\mathbb R} q(x)>0,$ and equality \eqref{3.1} holds if
and only if $q(x)\to\infty$ as $|x|\to\infty.$
\end{thm}
\renewcommand{\qedsymbol}{}
\begin{proof} Both relations \eqref{5.10} are proved in the same
way, and therefore we only consider, say, the second equality.
Below we use some properties of the function $F_2(\eta).$ It is
convenient to list these properties as a separate statement.
\end{proof}
\begin{lem}\label{lem5.3}
Under conditions \eqref{5.5}, the function $F_2(\eta)$ satisfies
the following relations:
\begin{enumerate}\item[1)] $F_2(\eta)\in\mathcal A
C^{\loc}(\mathbb R_+),$\quad $R_+=(0,\infty)$; \item[2)]
$F_2(\eta)>0$ for $\eta>0;$ \item[3)] $F_2'(\eta)>0$ for
$\eta>0.$\end{enumerate}\end{lem}

\renewcommand{\qedsymbol}{\openbox}
\begin{proof}
Property 2) is an obvious consequence of \eqref{5.5}. Further,
$$F_2'(\eta)=\frac{1}{r(x+\eta)}\int_x^{x+\eta}q(t)dt+q(x+\eta)\int_x^{x+\eta}\frac{dt}{r(t)}.$$
This equality together with \eqref{5.5} imply properties 1) and
3).
\end{proof}

\begin{lem}\label{lem5.4}
Let $\eta(x)=\alpha\hat d(x),$ $x\in\mathbb R,$ $\alpha\in(0,80].$
Then we have the inequalities
 \begin{equation}\label{5.14}
 \frac{r(x)}{\eta(x)}
 \left|\int_0^{\eta(x)}\left(\int_x^{x+s}\frac{r'(\xi)}{r^2(\xi)}\right)ds\right|\le
 \varkappa_1(x),\quad x\in\mathbb R,
\end{equation}
\begin{equation}\label{5.15}
\frac{1}{q(x)\eta(x)}\left|\int_0^{\eta(x)}\left(\int_x^{x+s}q'(\xi)d\xi\right)ds\right|\le
\varkappa_2(x),\quad x\in\mathbb R.\end{equation}
\end{lem}

\begin{proof}
Inequalities \eqref{5.14}--\eqref{5.15} are obvious. Say,
$$
\frac{r(x)}{\eta(x)}\left|\int_0^{\eta(x)}\left(\int_x^{x+s}\frac{r'(\xi)}{r^2(\xi)}d\xi\right)ds\right|
\le \frac{r(x)}{\eta(x)}\cdot\eta(x)\sup_{|s|\le 80\hat
d(x)}\left|\int_x^{x+s}\frac{r'(\xi)}{r^2(\xi)}d\xi\right|=\varkappa_1(x).$$
Let us now go to \eqref{5.10}. Let $\eta\ge0.$ The following
relations are obvious:
\begin{align}
\int_x^{x+\eta}\frac{d\xi}{r(\xi)}&=\int_0^\eta\frac{ds}{r(x+s)}=\frac{\eta}{r(x)}-\int_0^\eta\left(
\int_x^{x+s}\frac{r'(\xi)}{r^2(\xi)}d\xi\right)ds\nonumber\\
&=\frac{\eta}{r(x)}\left[1-\frac{r(x)}{\eta}\int_0^\eta\left(\int_x^{x+s}\frac{r'(\xi)}{r^2(\xi)}d\xi\right)
ds\right],\quad
x\in\mathbb R,\label{5.16}\end{align}
\begin{align}
\int_x^{x+\eta}q(t)dt&=\int_0^\eta
q(x+s)ds=q(x)\eta+\int_0^\eta\left(\int_x^{x+s} q'(\xi)
d\xi\right)ds\nonumber\\
&=q(x)\eta\left[1+\frac{1}{q(x)\eta}\int_0^\eta\left(\int_x^{x+s}q'(\xi)d\xi\right)
ds\right],\quad x\in\mathbb R.\label{5.17}\end{align}
 Denote
\begin{equation}
\begin{aligned}\label{5.18}
\delta(x)=\varkappa_1(x)+\varkappa_2(x),\qquad x \in\mathbb R,\\
\eta(x)=\hat d(x)(1+\delta(x)),\qquad x \in\mathbb R.
\end{aligned}
\end{equation}
Then for all $|x|\gg1,$ from \eqref{5.18}, \eqref{5.17},
\eqref{5.16}, \eqref{5.14} and \eqref{5.15}, it follows that
\begin{align}
F_2(\eta(x))&=\int_x^{x+\eta(x)}\frac{dt}{r(t)}\cdot\int_x^{x+\eta}q(t)dt\nonumber\\
&=\eta^2(x)\frac{q(x)}{r(x)}\cdot
\left[1-\frac{r(x)}{\eta(x)}\int_0^{\eta(x)}\left(\int_x^{x+s}\frac{r'(\xi)}{r^2(\xi)}d\xi\right)ds\right]\nonumber\\
&\quad \cdot
\left[1+\frac{1}{q(x)\eta(x)}\int_0^{\eta(x)}\left(\int_x^{x+s}q'(\xi)d\xi\right)ds\right]\nonumber\\
&
\ge(1+\delta(x))^2(1-\varkappa_1(x))(1-\varkappa_2(x))\nonumber\\
&\ge (1+2\delta(x))(1-\delta(x))=1+\delta(x)-2\delta^2(x)\ge
1.\label{5.19}
\end{align}

%\end{proof}

Since $F_2(0)=0,$ from \eqref{5.19} and \lemref{lem5.3}, it
follows that the equation $F_2(d)=1$ has a unique finite positive
solution. Denote it $d_2(x).$ From \eqref{5.19} and
\thmref{thm5.1}, we obtain the estimate
\begin{equation}\label{5.20}
d_2(x)\le \eta(x)=\hat d(x)(1+\delta(x),\qquad |x|\gg
1.\end{equation} Let now
\begin{equation}\label{5.21}
\eta(x)=\hat d(x)(1-\delta(x)\gg 1.\end{equation} Clearly,
$\eta(x)>0$ for all $|x|\gg1.$ The following relations are similar
to \eqref{5.19}:
\begin{align*}
F_2(\eta(x))&=\int_x^{x+\eta(x)}\frac{dt}{r(t)}\cdot
\int_x^{x+\eta(x)}q(t)dt\\ &=\eta^2(x)\cdot
\frac{q(x)}{r(x)}\left[1-\frac{r(x)}{\eta(x)}\int_0^{\eta(x)}\left(\int_x^{x+s}\frac{r'(\xi)}{r^2(\xi)}
d\xi\right)ds\right]\\
&\quad \cdot
\left[1+\frac{1}{q(x)\eta(x)}\int_0^{\eta(x)}\left(\int_x^{x+s}q'(\xi)d\xi\right)ds\right]\le(1-\delta(x))^2
(1+\varkappa_1(x))(1+\varkappa_2(x))\\
&=[1-2\delta(x)+\delta^2(x)][1+\varkappa_1(x)+\varkappa_2(x)+\varkappa_1(x)\varkappa_2(x)].
\end{align*}

It is easy to see that for all $|x|\gg1$, we have the
inequalities:\begin{align*} 1-2\delta(x)+\delta^2(x)&\le
1-\frac{5}{3}\delta(x)\\
\varkappa_1(x)\cdot \varkappa_2(x)&\le
\frac{\varkappa_1(x)+\varkappa_2(x)}{2}=\frac{\delta(x)}{2}\end{align*}
that allow us to continue the estimate
\begin{equation}\label{5.22}
F_2(\eta(x))\le\left(1-\frac{5}{3}\delta(x)\right)\left(1+\frac{3}{2}\delta(x)\right)\le
1-\frac{\delta(x)}{6}\le 1.\end{equation}

{}From \eqref{5.22} and \thmref{thm5.1} we obtain the inequality
\begin{equation}\label{5.23}
d_2(x)\ge \eta(x)=\hat d(x)(1-\delta(x)),\qquad |x|\gg
1.\end{equation} From \eqref{5.20} and \eqref{5.23} we obtain
\eqref{5.10}. Let us now go to \eqref{5.11}. These inequalities
are a consequence of \eqref{5.10}. Indeed, as above, we get
\begin{align}
\psi(x)&=\int_x^{x+d_2(x)}\frac{dt}{r(t)}=\frac{d_2(x)}{r(x)}-\int_0^{d_2(x)}\left(\int_x^{x+s}\frac{
r'(x)}{r^2(\xi)}d\xi\right)ds\nonumber\\
&=\frac{d_2(x)}{r(x)}\left[1-\frac{r(x)}{d_2(x)}\int_0^{d_2(x)}\left(\int_x^{x+s}\frac{r'(\xi)}{r^2(\xi)}
d\xi\right)ds\right]\nonumber\\
&\Rightarrow\quad \psi(x)\sqrt{r(x)q(x)}=\frac{d_2(x)}{\hat
d(x)}\cdot (1+\gamma(x)),\qquad x\in\mathbb R.\label{5.24}
\end{align}
Here \eqref{5.24} it easily follows that
$|\gamma(x)|\le\varkappa_1(x)$ for $|x|\gg 1.$ This proves
\eqref{5.11} and hence, in view of \eqref{2.12}, also
\eqref{5.12}. Let us verify \eqref{5.13}. Let us show that
$d(x)\le 80\hat d(x)$ for all $|x|\gg1$. Assume the contrary. This
means that $d(x)>\eta(x)=80\hat d(x)$ for some $|x|\gg1.$ In the
following relations, apart from the above assumption, we use
\eqref{2.19}, \eqref{2.26}, \eqref{2.22}, \thmref{thm5.1} and the
part of the theorem  that has already been proved:
\begin{align*}
1&=\int_{x-d(x)}^{x+d(x)}\frac{dt}{r(t)h(t)}\ge
\frac{1}{4e^2}\cdot
 \frac{1}{h(x)}\int_{x-d(x)}^{x+d(x)}\frac{dt}{r(t)}\\
 &\ge\frac{1}{80}\sqrt{r(x)q(x)}\left[2\frac{\eta(x)}{r(x)}-\int_0^{\eta(x)}\left(\int_x^{x+s}
 \frac{r'(\xi)}{r^2(\xi)}d\xi\right)ds\right.  \\
 &\quad +\int_0^{\eta(x)}\
 \left(\left.\int_{x-s}^x\frac{r'(\xi)}{r^2(\xi)}d\xi\right)ds\right]\ge
 2\cdot\left[1-\frac{1}{2}\cdot\frac{r(x)}{\eta(x)}\int_0^{\eta(x)}\left(\int_x^{x+s}\frac{r'(\xi)}{r^2(\xi)}
 d\xi\right)ds\right.\\
 &\quad+\frac{1}{2}\cdot\frac{r(x)}{\eta(x)}\int_0^{\eta(x)}\left.\left(\int_{x-s}^x\frac{r'(\xi)}{r^2(\xi)}d\xi
 \right)\right]\ge 2(1-\varkappa_1(x))>1.
\end{align*}
Contradiction. Hence
$$d(x)\le 80\hat d(x)\qquad \text{for}\qquad |x|\gg 1.$$
To get the lower estimate of $d(x)$ for $|x|\gg1,$ we use
\begin{align}
1&=\int_{x-d(x)}^{x+d(x)}\frac{dt}{r(t)h(t)}\le
\frac{4e^2}{h(x)}\int_{x-d(x)}^{x+d(x)}\frac{dt}{r(t)}\nonumber\\
&\le
80\sqrt{r(x)q(x)}\left[2\frac{d(x)}{r(x)}-\int_0^{d(x)}\int_x^{x+s}\frac{r'(\xi)}{r^2(\xi)}d\xi
ds\right.\nonumber\\
&\quad +\left.\int_0^{d(x)}\frac{r'(\xi)}{r^2(\xi)}d\xi
ds\right]\le160 \frac{d(x)}{\hat
d(x)}(1+\varkappa_1(x))\le320\frac{d(x)}{\hat d(x)}.\label{5.25}
\end{align}
Hence
\begin{equation}\label{5.26}
d(x)\ge \frac{\widehat{d(x)}}{320}\qquad\text{for}\qquad |x|\gg
1.\end{equation}

Choose $x_0\gg1$ so that for $|x|\ge x_0$ inequalities
\eqref{5.25} and \eqref{5.26} would hold together. Let
$$f(x)=\frac{d(x)}{\widehat{d(x)}},\qquad x\in [-x_0,x_0].$$
By \lemref{lem2.10}, the function $f(x)$ is positive and
continuous on $[-x_0,x_0]$ and therefore attains on this segment a
positive minimum $m$ and a finite maximum $M.$

Let $c\gg1$ be such that
$$c^{-1}\le\min\left\{\frac{1}{320},m\right\}\le \max\{80,M\}\le
c.$$ With such a choice of $c$, taking into account the fact
proven above, we obtain \eqref{5.13}. The remaining assertions of
the theorem follow from \eqref{5.12}--\eqref{5.13}.\end{proof}

We also need the following facts.

\begin{thm}\label{thm5.5} \cite[Ch.XI, \S6]{18}. Suppose that
conditions \eqref{1.2} and \eqref{2.1} hold, and, in addition,
\begin{equation}\label{5.27}
\int_{-\infty}^0\frac{dt}{r(t)}=\int_0^\infty\frac{dt}{r(t)}=\infty.\end{equation}
Then the relations
\begin{equation}\label{5.28}
v(x)\to0\qquad\text{as}\qquad x\to-\infty, \qquad
u(x)\to0\qquad\text{as}\qquad x\to\infty \end{equation} hold if
and only if \begin{equation}\label{5.29} \int_{-\infty}^x
q(t)\int_t^x\frac{d\xi}{r(\xi)}dt=\int_x^\infty
q(t)\int_x^t\frac{d\xi}{r(\xi)}dt=\infty,\quad x\in \mathbb
R.\end{equation}
\end{thm}

\begin{thm}\label{thm5.6} Suppose that conditions \eqref{1.2},
\eqref{2.1} and \eqref{5.28} hold, and, in addition, equation
\eqref{1.1} is correctly solvable in $L_p,$ $p\in(1,\infty).$ Then
equalities of \eqref{5.29} hold.\end{thm}

\begin{proof} The operator $G: L_p\to L_p,$ $p\in (1,\infty)$ is
 bounded by \thmref{thm2.18}. From \eqref{2.33} it follows that
 then so is the operator $G_2: L_p\to L_p,$ $p\in(1,\infty)$ (see
\eqref{2.31}). Then by \thmref{thm2.32} we have
 \begin{equation}\label{5.30}
 \sup_{x\in\mathbb
 R}\left(\int_{-\infty}^xv(t)^pdt\right)^{1/p}\left(\int_x^\infty
 u(t)^{p'}dt\right)^{1/p'}<\infty.
\end{equation}
Further, by \thmref{thm2.1} there
exist the limits
$$\lim_{x\to-\infty} v(x)=\varepsilon_1\ge 0,\qquad
\lim_{x\to\infty}u(x)=\varepsilon_2\ge0.$$ If here
$\varepsilon_1>0$ or $\varepsilon_2>0,$ then  \eqref{5.30} does
not hold. Hence $\varepsilon_1=\varepsilon_2=0.$ Then \eqref{5.29}
holds by \thmref{thm5.5}.\end{proof}

Let us now go to equation \eqref{5.1}. Denote by $S_p$ the set of
linear bounded operators acting from $L_p,$ $p\in(1,\infty)$, and
by $S_0^{(0)}$ the subset of $S_p$ consisting of the compact
operators. Thus writing $G\in S_p$ $(G\in S_p^{(0)})$ we mean that
the operator $G: L_p\to L_p$ is bounded (compact).

\begin{thm}\label{thm5.7} Let $G$ be the Green operator
corresponding to equation \eqref{5.1} (see \eqref{2.29}). Then,
regardless of $p\in(1,\infty)$ and depending on the numbers
$\alpha$ and $\beta$, the operator has the properties presented in
the following table.
\begin{table}[h]
%$$
\begin{equation}\label{5.31}
\begin{tabular}{|c|c|c|c|}
\hline
    $\alpha \ \setminus\ \beta$  & $\beta<0$ & $\beta=0$ & $\beta>0$  \\
   \hline
  $\alpha<0$ & $G\notin S_p$,
 & $G\in S_p,$\   $G\notin S_p^{(0)}$
 & $G\in S_p^{(0)}$
 \\
   \hline $\alpha=0$ &  $G\notin S_p$
 & $G\in S_p,$\   $G\notin S_p^{(0)}$
 & $G\in S_p^{(0)}$
  \\
    \hline  $\alpha>0$  & $G\in S_p^{(0)} $
 & $G\in S_p^{(0)} $
 & $G\in S_p^{(0)} $
 \\ \hline
\end{tabular}
%$$
\end{equation}
%\caption{Dependence $c$ on $s_i.$}
\end{table}
\end{thm}

\begin{proof}
Let us numerate the entries of matrix \eqref{5.31} in the usual
way and view them as instances of relations between $\alpha$ and
$\beta.$ We move along the rows of the matrix. Since in the case
of \eqref{5.1} the functions $r$ and $q$ are even, in all the
relations in the sequel, we only consider the case $x\ge0$
$(x\gg1)$.

\smallskip

\noindent $\underline{\text{Case}\ (1,1)\ (  \alpha<0,\ \beta<0)}$

Under conditions $(1,1)$ the hypotheses of \thmref{thm5.6} hold.
Therefore, the operator $G:L_p\to L_p,$ $p\in(1,\infty)$ can be
bounded only if \eqref{5.29} holds. In particular, we must have
the equality
\begin{equation}\label{5.32}
\infty=\int_0^\infty e^{\beta t}\left(\int_x^t
e^{-\alpha\xi}d\xi\right)dt\ \Rightarrow\
\beta\ge\alpha.\end{equation}

Below we consider cases \ a)~$\beta>\alpha$ and \
b)~$\beta=\alpha$ separately.

a) Let $\alpha<\beta<0$. In this case the hypotheses of
\thmref{thm5.2} hold, and therefore $B=\infty$ because
$\inf\limits_{x\in\mathbb R}q(x)=\inf\limits_{x\in\mathbb R}
e^{\beta|x|}=0.$ Thus $G\notin S$ by Theorems \ref{thm2.21} and
\ref{thm2.28}.

b) Let $\beta=\alpha<0.$ In this case $r=q,$ and one can compute
$h$ and $d$ directly. Thus the equation for $d_2(x)$ is of the
form (see \eqref{2.11})
$$1=\int_x^{x+d}e^{-\alpha\xi}d\xi\cdot\int_x^{x+d}e^{\alpha\xi}d\xi=\frac{(e^{|a|d}-1)(1-e^{-|\alpha|d})}
{\alpha^2},\quad d\ge0.$$ Hence $d_2(x)=c.$ To find $d_1(x)$, we
will first check that $d_1(x)\le x$ for all $x\gg1.$ Indeed, the
function $$F(x)=\int_0^x e^{-\alpha\xi}d\xi\cdot\int_0^x
e^{\alpha\xi}d\xi=\frac{e^{|\alpha|x}-e^{-|\alpha|x}-2}{\alpha^2}\to\infty$$
as $x\to\infty,$ and therefore $d_1(x)\le x$ for $x\gg1.$ Then
equation \eqref{2.11} for $d_1(x)$ is of the form:
$$1=\int_{x-d}^xe^{-\alpha\xi}d\xi\cdot\int_{x-d}^x
e^{\alpha\xi}d\xi=\frac{(1-e^{-|\alpha|d})(e^{|\alpha|d}-1)}{\alpha^2},\quad
d\ge0.$$ Hence $d_1(x)=d_2(x)=c.$ This easily implies the
equalities
$$\varphi(x)=ce^{|\alpha|\,|x|},\quad\psi(x)=ce^{|\alpha|\,|x|},\quad
h(x)=ce^{|\alpha|\,|x|},\quad d(x)=c.$$ Hence $B=\infty$ and
$G\notin S_p$, $p\in(1,\infty)$ by Theorems \ref{thm2.21} and
\ref{thm2.18}.

\smallskip

\noindent $\underline{\text{Case}\ (1,2)\ ( \alpha<0,\ \beta=0)}$

In this case $q(x)\equiv1,$ and therefore $G\in S_p$ by
\thmref{thm2.25}. We will use \thmref{thm5.2} to answer a more
subtle question on the inclusion $G\in S_p^{(0)}.$ It is easy to
see that in this case its hypotheses are satisfied, and equation
\eqref{3.1} does not hold because $q(x)\nrightarrow \infty$ as
$|x|\to\infty.$ Then $G\in S_p^{(0)}$, $p\in(1,\infty)$ by
\thmref{thm3.1}.

\newpage

\noindent $\underline{\text{Case}\ (1,3)\ ( \alpha<0,\ \beta>0)}$

In this situation conditions \eqref{1.2}--\eqref{1.3} hold and
$q(x)\to\infty$ as $|x|\to\infty.$ Then $G\in S_p^{(0)},$
$p\in(1,\infty)$ by Corollary \ref{cor3.8}.

\smallskip

\noindent $\underline{\text{Case}\ (2,1)\ ( \alpha=0,\ \beta<0)}$

Since $r\equiv1$ and $m(a)=0$, for any $a\in(0,\infty)$, we have
$G\notin S_p,$ $p\in(1,\infty)$ (see \thmref{thm2.24}).

\smallskip

\noindent $\underline{\text{Case}\ (2,2)\ ( \alpha=\beta=0)}$

Since $r\equiv q\equiv1,$ we have $G\in S_p,$ $p\in(1,\infty)$ by
\thmref{thm2.24}, and $G\notin S_p^{(0)}$, $p\in(1,\infty)$ by
Corollary \ref{cor3.9}.

\smallskip

\noindent $\underline{\text{Case}\ (2,3)\ ( \alpha=0,\ \beta>0)}$

We have $r\equiv1,$ $q(x)\to\infty$ as $|x|\to\infty.$ Hence $G\in
S_p^{(0)}\in(1,\infty)$ by Corollary \ref{cor3.8} or Corollary
\ref{cor3.9}.

\smallskip

\noindent $\underline{\text{Cases}\ (3,1); (3,2); (3,3)}$
$(\alpha>0$  and  $ \beta=0; \ \alpha>0$ {and} $ \beta>0,$ \
 respectively)

 All cases are treated in the same way. Clearly, $r^{-1}\in L_1,$
 $q>0$. Then $G\in S_p^{(0)}$, $p\in(1,\infty)$ by
 \thmref{thm3.13} or \thmref{thm3.14}.
 \end{proof}

 \section{Proofs of Otelbaev's Lemmas}

 In this section we present the proofs of Lemmas \ref{lem2.10},
 \ref{lem2.11}, \ref{lem2.15} and \ref{lem2.17} for the function
 $s(x)$ (see \eqref{2.19}). Assertions of such type (except for
 \lemref{lem2.17}, and with other auxiliary functions) were first
 applied by Otelbaev, and therefore we call them Otelbaev's Lemmas
 (see \cite{9}).

 \begin{proof}[Proof of \lemref{lem2.10}] Consider the function
 \begin{equation}\label{6.1}
 F(\eta)=\int_{x-\eta}^{x+\eta}\frac{dt}{r(t)\rho(t)},\qquad
 \eta\ge0.\end{equation}
 Clearly, the function $F(\eta)$ is continuous for $\eta\in[0,\infty;$
 $F(0)=0,$ $F(\infty)=\infty$ (see \eqref{2.10}), and
 $$F'(\eta)=\frac{1}{r(x+\eta)(\rho(x+\eta)}+\frac{1}{r(x-\eta)\rho(x-\eta)}>0.$$
 Therefore the second equation of \eqref{2.19} has a unique finite
 positive solution. Denote it by $s(x)$ and check that the
 function $s(x),$ $x\in\mathbb R$ is continuous. Towards this end,
 we show that the following inequality holds:
 \begin{equation}\label{6.2}
 |s(x+t)-sx)|\le|t|,\qquad |t|\le s(x),\qquad x\in\mathbb R.
 \end{equation}
 To check \eqref{6.2}, we have to consider two cases: \
 1)~$t\in[0,s(x)]$\ and \ 2)~$t\in[-s(x),0]$.

 They are treated in a similar way, and therefore below we only
 consider Case 1). Thus let $t\in[0,s(x)].$ Then we have the obvious
 inclusions
  $$[x-s(x),x+s(x)]\subseteq [(x+t)-(t+s(x)),(x+t)+(t+s(x))],$$
  $$[(x+t)-(s(x)-t),(x+t)+(s(x)-t)]\subseteq[x-s(x),x+s(x)],$$
 and therefore the following inequalities hold:

 \begin{equation*}
\left.\begin{array}{ll} \displaystyle{ 1=\int_{x-s(x)}^{x+s(x)}
\frac{d\xi}{r(\xi)\rho(\xi)} \le
\int_{(x+t-(t+s(x)))}^{(x+t)+(t+s(x))}
\frac{d\xi}{r(\xi)\rho(x)}},\\ \\ \displaystyle{
1=\int_{x-s(x)}^{x+s(x)} \frac{d\xi}{r(\xi)\rho(\xi)} \ge
\int_{(x+t)-(s(x)-t)}^{(x+t)+(s(x)-t)} \frac{d\xi}{r(\xi)\rho(x)}}
\end{array}\right\}\
\Rightarrow \ \left.\begin{array}{ll} \displaystyle{
s(x+t)\le t+s(x)}\\
\displaystyle{s(x+t)\ge s(x)-t}
\end{array}\right\}\ \Rightarrow \ \eqref{6.2}
\end{equation*}
{}From \eqref{6.2} it follows that $s(x),$ $x\in\mathbb R$ is
continuous.
\end{proof}

\begin{proof}[Proof of \lemref{lem2.11}] Let us rewrite
\eqref{6.2} in a different way:
\begin{equation}\label{6.3}
s(x)-|t|\le s(x+t)\le s(x)+|t|\quad\text{if}\quad |t|\le
s(x).\end{equation} Let $\xi=x+t.$ Then $t=\xi-x$, and in this
notation we obtain inequalities equivalent to \eqref{6.3}:
\begin{equation}\label{6.4}
s(x)-|\xi-x|\le s(\xi)\le s(x)+|\xi-x|\quad\text{if}\quad
|\xi-x|\le s(x).\end{equation}

Let $\varepsilon\in[0,1]$ and $|\xi-x|\le\varepsilon s(x).$ Then,
evidently, $|\xi-x|\le\varepsilon s(x)\le s(x)$, and \eqref{2.21}
follows from \eqref{6.4}:
$$s(x)-\varepsilon s(x)\le s(x)-|\xi-x|\le s(\xi)\le
s(x)+|\xi-x|\le s(x)+\varepsilon s(x).$$ Further, equalities
\eqref{2.23} are checked in the same way, and therefore below we
only consider the second one. We show that
$\varliminf\limits_{x\to\infty}(x-s(x))=\infty.$ Assume the
contrary. Then there exist a sequence $\{x_n\}_{n=1}^\infty$ such
that $x_n\to\infty$ as $n\to\infty$ and a number $c$ such that
$$x_n-s(x_n)\le c<\infty,\qquad n=1,2,\dots$$ From these
assumptions we obtain
$$1=\int_{x_n-s(x_n)}^{x_n+s(x_n)}\frac{dt}{r(t)\rho(t)}\ge
\int_c^{x_n}\frac{dt}{r(t)\rho(t)}\to\infty\quad\text{as}\quad
n\to\infty $$ (see \eqref{2.10}). Contradiction. Hence
$$\varliminf_{x\to\infty}(x-s(x))=\infty \ \Rightarrow\ \infty\le
\varliminf_{x\to\infty}(s-s(x))\le
\varlimsup_{x\to\infty}(x-s(x))\le\infty\ \Rightarrow$$
$$\varliminf_{x\to\infty}(x-s(x))=\varlimsup_{x\to\infty}(x-s(x))=\infty\
\Rightarrow\ \eqref{2.23}$$
\end{proof}

\begin{proof}[Proof of \lemref{lem2.15}] The assertion of the
lemma immediately follows from Lemmas \ref{lem2.10}, \ref{lem2.11}
and \ref{lem2.13}.
\end{proof}

\begin{proof}[Proof of \lemref{lem2.17}]
Below for $t\in [x,x+s(x)]$, we use \thmref{thm2.1}, \eqref{4.34},
\eqref{2.17} and \eqref{2.19}:
$$\frac{v'(t)}{v(t)}=\frac{1+r(t)\rho'(t)}{2r(t)\rho(t)}\le\frac{1}{r(t)\rho(t)},\quad
t\in[x,x+s(x)]\ \Rightarrow$$
$$ln\frac{v(x+s(x))}{v(x)}\le\int_x^{x+s(x)}\frac{dt}{r(t)\rho(t)}<\int_{x-s(x)}^{x+s(t)}\frac{dt}
{r(t)\rho(t)}=1.$$ Similarly,
$$ln \frac{v(x)}{v(x-s(x))}\le\int_{x-s(x)}^x\frac{dt}{r(t)\rho(t)}<\int_{x-s(x)}^{x+s(x)}\frac
{dt}{r(t)\rho(t)}=1.$$
 This gives the  inequalities of
\eqref{2.27}, for example:
$$e^{-1}\le\frac{v(x-s(x))}{v(x)}\le\frac{v(t)}{v(x)}\le\frac{v(x+s(x))}{v(x)}\le
e,\quad |t-x|\le s(x).$$ Inequalities \eqref{2.27} for the
function $\rho$    is a consequence of  the following inequalities
for $u$ and $v$:
$$c^{-1}\le \frac{\rho(t)}{\rho(x)}=\frac{u(t)}{u(x)}\
\frac{v(t)}{v(x)}\le c,\qquad |t-x|\le s(x).$$
\end{proof}

\end{document}